%% file: interpolation-hd210422.tex
\documentclass{siamart171218}

\input{ex_shared}

\ifpdf
\hypersetup{
  pdftitle={Kernel Interpolation of High Dimensional Scattered Data},
  pdfauthor={S.-B. Lin, X. Chang, and X. Sun}
}
\fi

\begin{document}

\maketitle

\begin{abstract}
 Data sites selected from modeling high-dimensional problems often appear scattered in non-paternalistic ways. Except for  sporadic-clustering at some spots, they become relatively far apart as the  dimension of the ambient space grows. These features defy any theoretical treatment that requires local or global quasi-uniformity of  distribution of data sites. Incorporating a recently-developed application of integral operator theory in machine learning, we propose and study in the current article a new framework to analyze kernel interpolation of high dimensional data, which features bounding stochastic approximation error by  the spectrum of the underlying 
 kernel matrix. Both theoretical analysis and numerical simulations show that spectra of kernel matrices are  reliable and stable  barometers  for gauging the performance of kernel-interpolation methods for high dimensional data.
\end{abstract}



\section{Introduction}
Let $\mathcal X$ be a compact domain in $\mathbb R^d$ with Lipschitz boundary. Let $K(\cdot,\cdot):{\mathcal X}\times {\mathcal X} \rightarrow \R$ be a continuous, symmetric and strictly  positive-definite kernel. Suppose that a data set $D:=\{(x_i,y_i)\}_{i=1}^m$ is given, in which $\Xi:=\{x_i\}_{i=1}^m $ are $m$ scattered points from $\mathcal{X}$, and $\{y_i\}_{i=1}^m$
are values of a target function $f$ taken on $\Xi $. In employing a kernel method to model a  real world problem,  one designs or adopts an algorithm to select an $f_D \in K_\Xi:={\rm span}\{K_{x_j}\}_{i=1}^m$, which represents ``faithfully" the  target function $f$ on $\mathcal{X}$. Here $K_{x_j}$ denotes the function: $\mathcal{X} \ni x \mapsto K(x_j, x)$. While the selection of an algorithm is subject to practical constraints and (possibly) subjective bias,  and criteria for the faithfulness of the representation are up to improvising, veracity,  and (even) debate, the approximation capability of the subspace $K_\Xi$ is always at the core of every theoretical consideration. In a reproducing kernel Hilbert Space (RKHS) setting (often referred to as a native space in the approximation theory community), the best approximation from the subspace $K_\Xi$ is achieved via interpolation. That is, one chooses $f_D \in K_\Xi, $ such that $f_D(x_j)=y_j, \; j=1,\ldots,m,$ which can be precisely written as the following:
\begin{equation}\label{interpolation-est}
f_D= \sum_{i=1}^m a_i K_{x_i},\quad \mbox{in which}\quad (a_1,\dots,a_m)^T=\mathbb K^{-1}y_D.
\end{equation}
Here $\mathbb K=(K(x_i,x_j))_{i,j=1}^m$ denotes the interpolation matrix (also called kernel matrix), and $y_D=(y_1,\dots,y_m)^T$. 
For some radial basis kernels, such as thin plate splines, approximation power can transcend a native space barrier so that a ``near best" approximation order be realized for functions from a larger (than the native space) RKHS.  For readers who are interested in the above native space approximation narrative, we make reference to \cite{madych-nelson, Narcowich1991,  Narcowich2002, Narcowich2004, Narcowich2006, Narcowich2007, Park1991,  Schaback1995, Schaback2000, Wendland2004}, and the bibliographies therein.

Hangelbroek et al (\cite{hangel-narc-sun-ward,hangel-narc-ward-1,hangel-narc-ward-2}) have recently made significant advancement in expanding the approximation power of interpolation beyond the native space setting, a gist of which will be summarized as follows. Let
\[ h_\Xi:=\max_{x\in\mathcal X}\min\limits_{1 \le j \le m} d(x_j,x), \quad q_\Xi=\frac12\min\limits_{j\neq k}d(x_j,x_k).
\]
The former is the Hausdorff distance between the point set $\Xi $ and $\mathcal X$, but is more commonly  referred to in the literature as  mesh norm or fill-distance; the latter is the separation radius (or half of the minimal separation) of the point set $\Xi $. If there is a constant $1 \le C_d$ depending only on $d$, such that $h_\Xi/q_\Xi \le C_d $, then we say that 
the point set $\Xi $ is quasi-uniformly distributed in $\mathcal X$. Global or local quasi-uniformity of a data set $\Xi$ is a crucial analytical tool  for meshless kernel methods to achieve their approximation goals. In particular, approximation orders of meshless kernel methods are mostly given in terms of  $h_\Xi$.  

Let $\Omega$ be a $d$-dimensional Riemannian manifold. Let $\Xi^*$ be a discrete subset of $\Omega$ that is quasi-uniformly distributed in $\Omega$. Let $\chi_\xi$ be the Lagrange interpolating function headquartered at $\xi$ and associated with the surface splines or the Mat\'ern kernel. That is $\chi_\xi (\zeta) = \delta_{\xi, \zeta}, \; \xi, \zeta \in \Xi^*,$  where $\delta_{\xi, \zeta}$ is the Kronecker delta. Hangelbroek, Narcowich, and Ward \cite{hangel-narc-ward-1} established the following remarkable inequality:
\begin{equation}\label{hangel}
|\chi_\xi(x)| \le C(d) \exp \left[ - \nu(d) \frac{{\rm dist} (x, \xi)}{h_{\Xi^*}}\right], \quad x \in \Omega.
\end{equation}
Here  $C(d), \nu(d)>0$ are constants depending only on $d$ and the underlying kernels. However, both constants grow at an exponential rate with respect to $d$. They further showed that the interpolation operator is bounded from $C_p(\Omega)$ to itself, where $C_p(\Omega)$ denotes the totality of continuous functions on $\Omega$ with polynomial growth at infinity. In a follow up article, Hangelbroek et al \cite{hangel-narc-sun-ward} proved that the $L^2$-projector is bounded under the $L^\infty$-norm. 
Leveraging the exponential decay of $\chi_\xi$ away from the base point $\xi$ as shown in inequality \eqref{hangel}, Hangelbroek, Narcowich, and Ward \cite{hangel-narc-ward-2} and Fuselier et al \cite{fuselier-hangel-narc-ward} articulated
the notion ``local density function"  in which Lagrange interpolating functions are built on data sets whose cardinality are of logarithmic orders, and yet the interpolation scheme still achieves desirable approximation orders.  This has vastly reduced computational complexity, and enhanced the efficiency of many meshless methods in solving partial differential equations on domains of  relatively low dimensions; see \cite{hangel-narc-rieger-ward}.

Modern-day data scientists are encountering an onslaught of real world problems in which massive data sets are involved. In many cases,  data seem extremely disorganized and even outright chaotic, which not only poses challenges but also provides opportunities for data scientists to figure out ways to store, communicate and analyze them. A persisting challenge stems from experiences in dealing with the enormous number of features (variables) data sets exhibit.  For example,  microarrays for gene expression \cite{Baldi2011}  contain thousands of samples, each of which  in turn has tens of thousands of genes. Another well-known example is the natural image data set - ImageNet \cite{Deng2009}, which gathers about 14 million natural images classified in more than 20,000 categories. Each image has the original resolution with \(469\times 387=1,823,003\) pixels (dimensions). Our numerical simulations show that high dimensional data sites  may exhibit sporadic-clustering  at some spots, but are mostly scattered in non paternalistic ways and relatively far apart from each other, which defies any attempt to analyze them using the likes of local density functions.   

Suppose that mass is uniformly distributed on $[0,1]^d$, the unit cube in $\R^d$. Then for any fixed $0< \epsilon < \frac12,$ and a sufficiently large $d$,  the law of large numbers shows that the mass of $[0,1]^d$ is mostly concentrated in an $\epsilon$-neighbourhood of the hyperplane $\mathcal{L}: x_1 + \cdots + x_d = \frac{d}{2}$, which happens to be the orthogonal bisector of the main diagonal of $[0,1]^d$ (which has length $\sqrt{d}$).  Meanwhile the cube $[\epsilon, 1-\epsilon]^d$ has volume $(1-2\epsilon )^d$, which approaches zero exponentially fast with $d$. Thus, the mass of $[0,1]^d$ is mostly concentrated on the intersection of $\epsilon$-neighbourhood of the hyperplane $\mathcal{L}$ and the set $[0,1]^d \setminus [\epsilon, 1-\epsilon]^d$.  Figure \ref{fig:curse_dimensionality} depicts the situation for $d=3$.  

\begin{figure}[ht]
	\centering
	\includegraphics[scale=0.13]{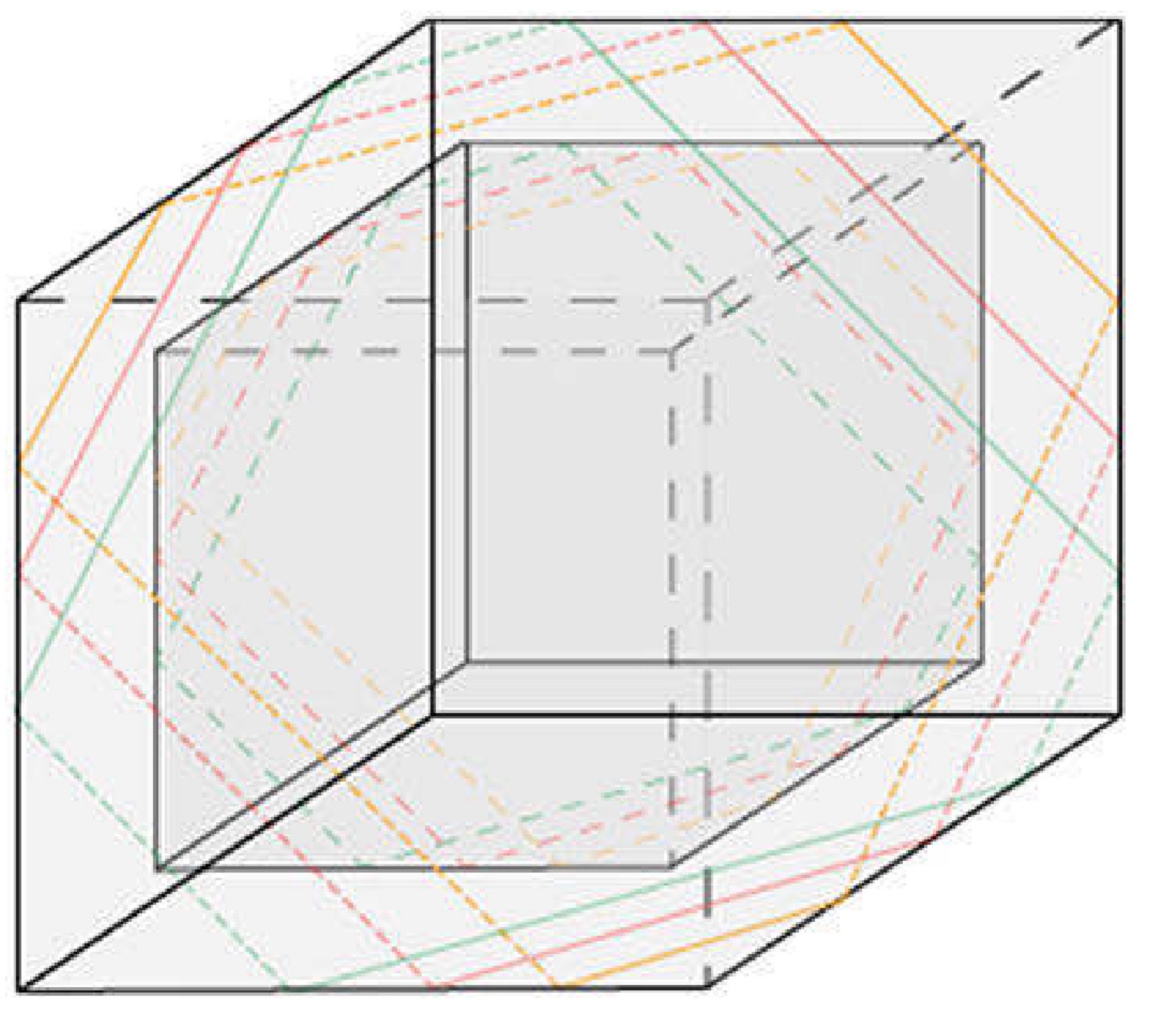}
	\caption{Concentration of mass on the unit cube.}\label{fig:curse_dimensionality}
\end{figure}

Figure \ref{subfig:sr} exhibits the increase of the separation radius $q_\Xi$ with $d$ for random samplings of $500$ points from $[0,1]^d$ (according to the uniform distribution). In Figure \ref{subfig:sr}, each red dot indicates the mean values (in $10$ trials) of separation radius for each dimension $d$ in the range $2 \le d \le 100$. Each red dot is accompanied with an error bar, indicating the range of variation of separation radius from these trials as determined by the double standard deviation. To further demonstrate the potency of our main methodology undertaken here, we have designed and carried out the following large scale numerical simulation.  For each given dimension $d$ in the range $2 \le d \le 100$, we first randomly select (according the uniform distribution on $[0,1]^d$) $500$ points $x_1,\ldots, x_{500}.$ We then calculate
the condition number (associated with the $\ell^2$-norm) of the corresponding kernel matrix $\mathbb{K}:=(G(x_i,x_j))\in\mathbb{R}^{500\times500}$, where $G(x,y)=\exp\{-\|x-y\|^2/2\}.$  Figure \ref{subfig:cn} shows that  the condition  number of $\mathbb{K}$ decays exponentially fast to one as $d$ increases from  $2$ to $ 100$. 
Similar to Figure 2 (a), each red dot indicates the mean value (in 10 trials) of the condition number of $\mathbb{K}$ for the corresponding dimension $d$, and the accompanied error bar indicates the double standard deviation. \footnote{We have repeated the same simulation many times, and got more or less the same result. This has motivated us to establish probabilistic lower bounds of $q_X$ in terms of $m$ and $d$; see Lemma \ref{Lemma:separation-for-gaussian} (in Appendix C). }

\begin{figure*}[!htb]
	\centering
	\subfigure[]{\includegraphics[scale=0.3]{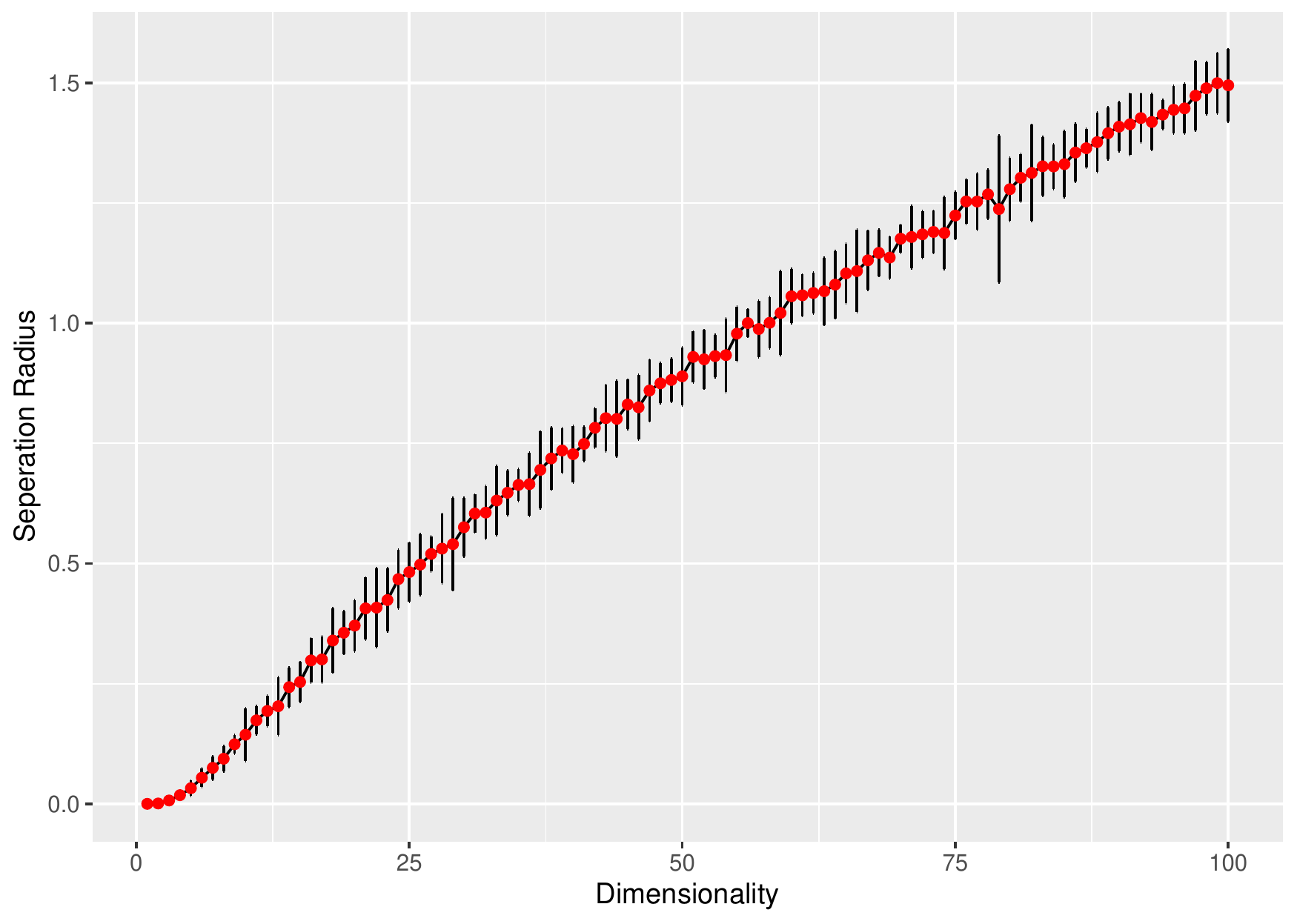}
		\label{subfig:sr}
	} 
	\subfigure[]{\includegraphics[scale=0.3]{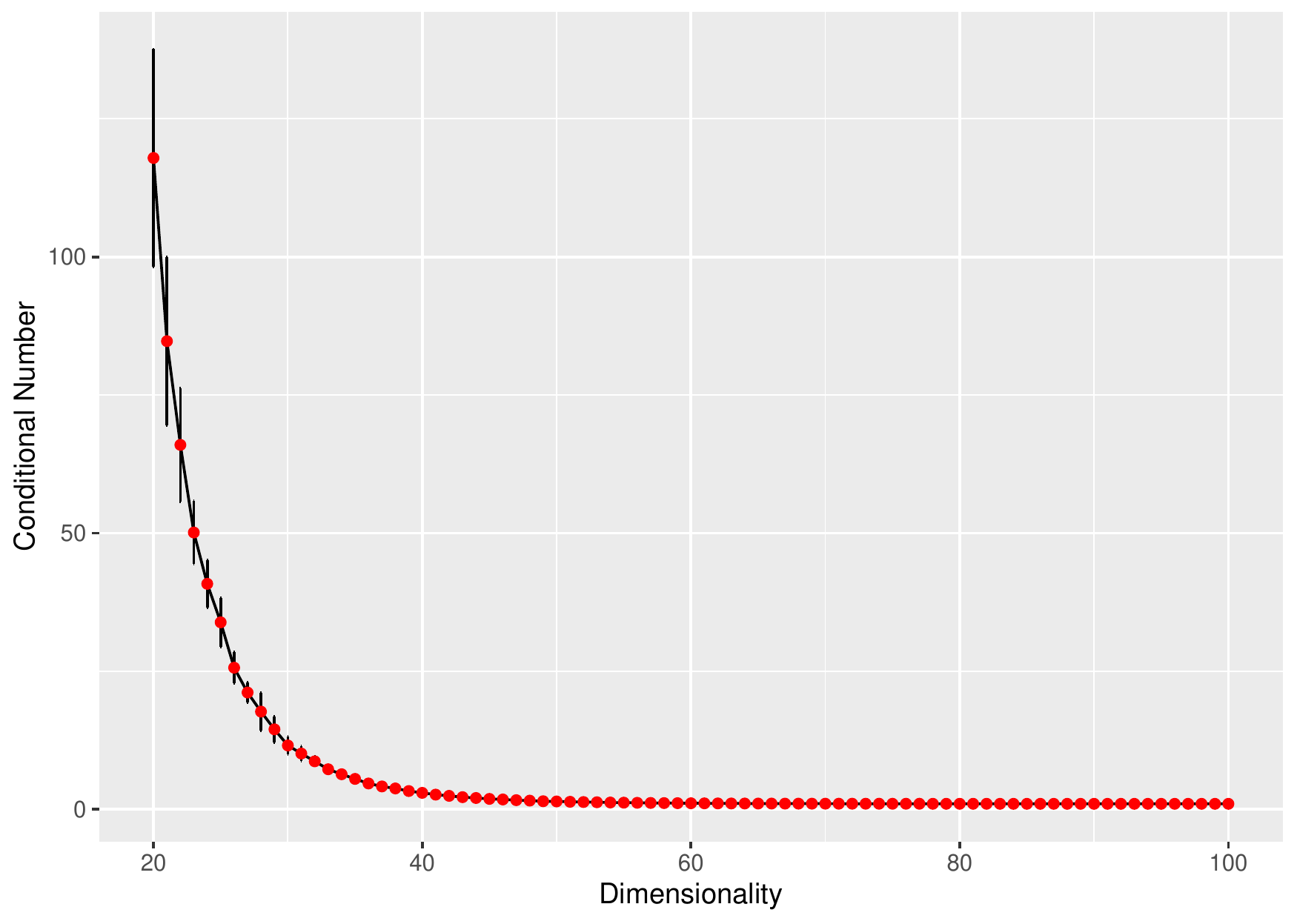}
		\label{subfig:cn}}
	\caption{The change of random separation radii and condition numbers as dimension increases.}
	\label{fig:motivation_example}
	\vspace{0.2in}
\end{figure*}

The main goal of the current paper is to propose and study a new stochastic framework for kernel interpolation of  high dimensional data.  Inspired by our numerical simulation results and Peetre's idea \cite{peetre} in the study of interpolation of operators, we introduce two quantities (parameters) in Equations (\ref{Def.Q}) and (\ref{Def.R}), and use them to bound the error of stochastic approximation for the underlying kernel interpolants. By fine-tuning the two parameters, we observe a $K$-functional of a hybrid (continuous and discrete) nature. We think the mathematical lineage interesting, and will investigate deeper connections in the future.

Based on the recently-developed integral operator theory \cite{Smale2007,Lin2017,Lin2019}, we first express the approximation error as the difference  between integral operators and the corresponding empirical discretizations.  We then formulate the difference in terms of the spectrum of the kernel matrix. Finally, we employ pertinent concentration inequalities \cite{Pinelis1994} in Banach spaces to derive the desired error estimates (Theorems \ref{Theorem:Stability-via-eigenvalue} and \ref{Theorem:native-barrier-without-noise-eigenvalue}). Working behind the scene are spectrum estimates of kernel matrix \cite{Ball1992, Narcowich1991,Schaback1995,Wendland2004,Elkaroui2010,Belkin2018,Liang2020}, of which we mention particularly that in Ball's estimate \cite{Ball1992} of the smallest eigenvalue of distance matrices in terms of the minimal separation between the data sites, the constant grows algebraically with dimension. 

To demonstrate the versatility of our method, we derive stochastic approximation errors for kernel interpolation under three different computing environments. Firstly, we establish a close relationship between approximation error and spectrum of the kernel matrix for noise-free data, that is, $D=\{(x_i,y_i)\}_{i=1}^m$ with $y_i=f^*(x_i)$ for $f^*\in\mathcal H_K$, where $\mathcal H_K$ is the RKHS associated with $K$. Secondly, we study the performance of kernel interpolation with the presence of noise. Under the circumstance, the data $\{y_i\}_{i=1}^m$ are of the form: $y_i=f^*(x_i)+\varepsilon_i$ in which $f^*\in\mathcal H_K$ and  $\varepsilon_i$ indicate some white noise. Our result shows that there is a trade-off between accuracy of approximation and stability of the underlying algorithm in terms of kernel selections. Finally, we investigate the approximation capability of kernel interpolation beyond the native space setting, assuming that data come from a target function outside of the native space, which we will refer to as ``trans-native space data". This  is figuratively called ``conquering  the native space barrier'' in \cite{Narcowich2007}.

The rest of paper is organized as follows. In Section \ref{Sec.operator.difference}, we outline the integral operator approach, which is a theoretical pillar of the current article,  and establish the relationship between the approximation error and the norm estimate of the corresponding integral operator.  In Section \ref{Sec.Random}, we formulate the approximation error of kernel interpolation in terms of the spectrum of kernel matrix. In Section \ref{sec.spectrum}, we give spectrum estimates for some widely used kernels and remarks on the implication for the ensuing sampling and interpolating operations. In Section \ref{Sec.Experiment}, we report some  numerical simulation results.  In Section \ref{Sec.proof}, we present the proofs of our results.
Not to distract readers' attention from the main narrative, we collect some frequently-used results in three appendices for easy referencing.

\section{Error Analysis for Kernel Interpolation via Finite Differences of Operators}\label{Sec.operator.difference}

This section features a novel integral operator approach  to analyze the approximation performance of  kernel interpolation. Prototypical ideas of this approach already appeared in \cite{Smale2004,Smale2005,Lin2017}.
Let $\rho_X$ be a probability measure on $\mathcal X$.  Denote by $L^2_{\rho_X}$ the space of $\rho_X$-square-integrable functions endowed with norm $\|\cdot\|_\rho$. Define the 
integral operator $L_K:\mathcal H_K\rightarrow\mathcal H_K$ by
\begin{equation}\label{integral-1}
L_K(f) =\int_{\mathcal X} K_x f(x)d\rho_X, \qquad f\in {\mathcal
	H}_K.
\end{equation}
Let $\mathcal L_K: L_{\rho_X}^2\rightarrow L_{\rho_X}^2$ be the integral operator defined by
$$
\mathcal L_Kf=\int_{\mathcal X}f(x)K_xd\rho_X,   f\in L_{\rho_X}^2.
$$
We have of course that $\mathcal L_Kf=  L_Kf$ for  $f\in\mathcal H_K$.
However, one can neither consider $\mathcal L_K$ an extension of $L_K$ from $H_K$ to $L_{\rho_X}^2$ nor $L_K$ a restriction of $\mathcal L_K$ on $H_K$, as the norms of two spaces are not equivalent (when restricted to $L_K$). In fact, for an arbitrary $f\in L_{\rho_X}^2$, we have (see \cite{Caponnetto2007})
that $
\|f\|_\rho=\|\mathcal L_K^{1/2}f\|_K$, where $\mathcal L_K^r\ (r >0)$  is
defined by spectral calculus and $\|f\|_K$ denotes the norm of the RKHS $\mathcal H_K$.   
Let $S_{D}:\mathcal H_K\rightarrow\mathbb R^{m}$ be the sampling
operator defined by
$$
S_{D}f:=(f(x_i))_{i=1}^m.
$$
Its scaled adjoint $S_D^T:\mathbb R^{m}\rightarrow
\mathcal H_K$   is
given by
$$
S_{D}^T{\bf c}:=\frac1{m}\sum_{i=1}^{m}c_iK_{x_i},\qquad {\bf
	c}:=(c_1,c_2,\dots,c_{m})^T
\in\mathbb
R^{m}.
$$
Then, we have
\begin{equation}\label{operator-matrix}
\frac1m\mathbb K=S_DS_D^T.
\end{equation}
This together with (\ref{interpolation-est}) implies
\begin{equation}\label{operator-rep1}
f_D=S_D^T(S_DS_D^T)^{-1}y_D.
\end{equation}
We carry out error analysis for the following three computing environments: (i) noise-free  data in the native space setting; (ii) noisy-data in the native space setting; (iii) trans-native space data. En route we will frequent some definitions and theorems pertaining to operator theory, which  we collect in  Appendix A
for easy referencing.

\subsection{Kernel interpolation of noise-free data}
In this subsection, we study the approximation error of $f_D$ defined by (\ref{interpolation-est}) when the data are noise-free, i.e., there exists a  function $f^*\in L^2_{\rho_X}$ such that
\begin{equation}\label{model-1}
y_i=f^*(x_i).
\end{equation}
Define the empirical version of the integral operator $L_K$ to be
$$
L_{K,D}f:=S_D^TS_Df=\frac1{m}\sum_{i=1}^mf(x_i)K_{x_i}.
$$
Since $K$ is strictly positive definite,  $L_{K,D}$ is a  positive operator  of rank   $m$.  Denote by $\{(\sigma_\ell^D,\phi_\ell^D)\}_{\ell=1}^\infty$   the normalized eigen-pairs of $L_{K,D}$ with
$\sigma_1^D\geq\sigma_2^D\geq\dots\geq\sigma_m^D>0$ and $\sigma_{m+j}^D=0$ for $j\geq1$. 
In the following proposition, we show that $f_D-f^*$ is in the null space of the operator $L_{K,D}$.
\begin{proposition}\label{Proposition:null-space}
If $D=\{(x_i,y_i)_{i=1}^m$ satisfies  \eqref{model-1}, then $f^*-f_D$ is in the null space of the operator $L_{K,D}$, i.e.
\begin{equation}\label{Null-space}
L_{K,D}(f^*-f_D)=0.
\end{equation}
\end{proposition}

  Proposition \ref{Proposition:null-space} implies for any $\nu\geq 1/2$ that
\begin{eqnarray}\label{Interpolation-property}
    \|L_{K,D}^{\nu}(f^*-f_D)\|_K^2
    &=&
    \langle L_{K,D}^{\nu}(f^*-f_D), L_{K,D}^{\nu}(f^*-f_D)\rangle_K\nonumber\\
    &=&
    \langle L_{K,D}(f^*-f_D), L_{K,D}^{2\nu-1}(f^*-f_D)\rangle_K
    =0.
\end{eqnarray}
To quantitatively describe the approximation performance of $f_D$, some regularity of the target function $f^*$ should be imposed. We adopt the widely used regularity assumption \cite{Caponnetto2007,Lin2017,Lin2018} via the integral operator $\mathcal L_K$.
\begin{equation}\label{regularity-assump}
f^*=\mathcal L_K^r h^*,\qquad \ h^*\in L_{\rho_X}^2,\ r\geq0.
\end{equation}

Let $\{(\sigma_\ell,\phi_\ell)\}_{\ell=1}^\infty$ be the normalized eigen-pairs of $  L_{K}$ with
$\sigma_1\geq\sigma_2\geq\dots\geq0$.   Then Aronszajn thoerem \cite{aronszajn1950theory} shows 
 that  $\{(\sigma_\ell,\phi_\ell/\sqrt{\sigma_\ell})\}_{\ell=1}^\infty$ consists an normalized eigen-paris of $\mathcal L_K$. Furthermore,
the Mercer expansion \cite{aronszajn1950theory} shows 
$$
    K(x,x')=\sum_{\ell=1}^\infty \phi_\ell(x)\phi_\ell(x')
           =\sum_{\ell=1}^\infty \sigma_\ell\frac{\phi_\ell(x)}{\sqrt{\sigma_\ell}}\frac{\phi_\ell(x')}{\sqrt{\sigma_\ell}}
$$
Then
the regularity condition (\ref{regularity-assump}) is equivalent to
\begin{equation}\label{Regularity-1111}
      f^*=\mathcal L_K^rh^*=\sum_{\ell=1}^\infty \sigma^{r-1/2}_\ell\langle  h^*,\phi_\ell/\sqrt{\sigma_\ell}\rangle_\rho  \phi_\ell.
\end{equation}
This means that the parameter $r$  in (\ref{regularity-assump}) determines the smoothness of the target functions. In particular, (\ref{regularity-assump}) with $r=1/2$ implies $f^*\in\mathcal H_K$,
(\ref{regularity-assump}) with  $r<1/2$ implies $f^*\notin\mathcal H_K$ and  (\ref{regularity-assump}) with $r=0$ yields $f_\rho\in L_{\rho_X}^2$. 
 Generally speaking, the larger value of $r$ is, the smoother the function $f^*$ is. There is a rich literature devoted to characterizing the smoothness of functions in terms of the decay rate of eigenvalue sequence of the associated compact positive operators;  see Weyl \cite{weyl}, K\"uhn\cite{kuhn2}, Reade \cite{reade-2}, and the references therein. 

As the regularity is described via $L_K$ while $f_D-f^*$ is in the null space of $L_{K,D}$, we introduce two quantities to measure the difference between $L_K$ and $L_{K,D}$.  
For any $\lambda>0$, write
\begin{eqnarray}
\mathcal Q_{D,\lambda} &:=& \left\|(L_{K,D} + \lambda I)^{-1/2}(L_{K}+\lambda
I)^{1/2} \right\|,   \label{Def.Q}\\
\mathcal R_D &:=&
\|L_{K,D}-L_K\|_{HS}, \label{Def.R}
\end{eqnarray}
where $\|A\|$ and $\|A\|_{HS}$ denote, respectively, the spectral norm and  Hilbert-Schmidt norm of the   operator $A$.
We impose the interpolation condition (\ref{model-1}), and use    two quantities: $\mathcal Q_{D,\lambda}$ and $\mathcal R_D$  to bound the approximation error of kernel interpolants. A suitable combination of the two quantities above can be loosely interpreted  as a hybrid (discrete and continuous) version of the ``$K$-functional'' introduced by Peetre \cite{peetre} in the context of  operator interpolations.
Based on the above preliminaries, we are in a position to present our error estimate in the following theorem.

\begin{theorem}\label{Theorem:noise-less-operator}
Let $0<u\leq 1/2$.	Suppose that $D=\{(x_i,y_i)\}_{i=1}^m$ satisfies (\ref{model-1}) and  that  $f^*$ satisfies (\ref{regularity-assump}) with $r\geq 1/2$. Then we have
	\begin{equation}\label{noise-less-operator}
	\|L_K^u(f_D-f^*)\|_K
	\leq
	\min_{\lambda>0} \left\{\begin{array}{ll}
	\lambda^{r+u-1/2}\mathcal Q_{D,\lambda}^{2r+2u-1}\|h^*\|_\rho, & \quad 1/2\leq r\leq 1,\\
	(r-1/2)\kappa^{r-3/2}\|h^*\|_\rho\lambda^u\mathcal Q_{D,\lambda}\mathcal R_D,  & \quad r>1,
	\end{array}
	\right.
	\end{equation}
where $\kappa=\sqrt{\sup_{x,x'\in\mathcal X}K(x,x')}$.
\end{theorem}

  Since $\mathcal X$ is compact, it is easy to check that $\kappa<\infty$, which implies $\|f\|_\infty\leq \kappa\|f\|_K$. This together with $\|f\|_\rho=\|\mathcal L_K^{1/2}f\|_K$ shows that \eqref{noise-less-operator} includes error estimates under both  $L_{\rho_X}^2$ norm and $L^\infty$ norm for $u=1/2$ and $u=0$, respectively.   In particular, 
 Theorem \ref{Theorem:noise-less-operator} with $u=1/2$ and $r=1/2$ yields the following corollary.
 
\begin{corollary}\label{Corollary:rho-noiseless-operator}
Suppose that $D=\{(x_i,y_i)\}_{i=1}^m$ satisfies (\ref{model-1}) and   $f^*\in\mathcal H_K$. Then we have
$$
	\|f_D-f^*\|_\rho
	\leq
	\min_{\lambda>0}\lambda^{1/2}\mathcal Q_{D,\lambda}\|h^*\|_\rho. 
$$
\end{corollary}  
 
As shown in the above corollary, if $f^*\in\mathcal H_K$, a standard assumption in  approximation theory \cite{hangel-narc-sun-ward,Narcowich2002,Narcowich2004} and learning theory \cite{Caponnetto2007,Smale2007,Lin2017}, the approximation rate depends only on   $\mathcal Q_{D,\lambda}$, the difference between $L_K$ and its discretion counterpart. Generally speaking, the mentioned difference in $\mathcal Q_{D,\lambda}$ decreases with respect to $\lambda$, making the minimization in Corollary \ref{Corollary:rho-noiseless-operator} be well definied.
Theorem \ref{Theorem:noise-less-operator} with $u=0$ and $r>1$ and $\|L_K\|\leq \kappa$ and $\lambda=1$ can also yield the following error estimate under the $L^\infty(\mathcal X)$ norm. 
 
\begin{corollary}
Suppose that $D=\{(x_i,y_i)\}_{i=1}^m$ satisfies (\ref{model-1}) and  that  $f^*$ satisfies (\ref{regularity-assump}) with $r>1$. Then we have
$$
	\|f_D-f^*\|_\infty
	\leq
	 (r-1/2)(\kappa+1)^{r}\|h^*\|_\rho\mathcal R_D .
$$
\end{corollary} 

All the above results show that the approximation error of kernel interpolation depends heavily on the differences of operators $L_K$ and $L_{K,D}$ and is independent of the dimension of input space $\mathcal X$. This makes our error analysis totally different from the classical results  \cite{Narcowich2002,Narcowich2004,hangel-narc-rieger-ward,hangel-narc-ward-1,Wendland2004} which introduced the distribution of scattered data such as the mesh norm, separation radius and mesh ratio to quantify the error.

\subsection{Kernel interpolation of noisy-data}
In this part, we study the approximation performance of kernel interpolation    when the data are noisy, that is, there exists an $f^*$ satisfying (\ref{regularity-assump}) with $r\geq 1/2$  such that
\begin{equation}\label{model-2}
y_i=f^*(x_i)+\varepsilon_i,
\end{equation}
where $\varepsilon_i$  satisfies $E[\varepsilon_i]=0$ and $|\varepsilon_i|\leq \gamma$ for some $\gamma\geq 0$. It should be mentioned that our noisy model (\ref{model-2}) is different from the classical setting \cite{Wendland2004,hesse2017radial} that is only available to  extremely  small noise. In fact, the magnitude of noise  in \eqref{model-2} can be comparable  with $\|f^*\|_\infty$. Instead, we   assume the noise to be randomly white noise, which is standard in statistics and learning theory \cite{Caponnetto2007,Liang2020,lin2021distributed}.

The  approximation error analysis for (\ref{model-2}) is much more sophisticated than the noise-free model (\ref{model-1}) which requires the kernel matrix to be well-conditioned. To gauge the effect of noise, we introduce the following quantity:
\begin{eqnarray} \label{Def.P}
\mathcal P_{D,\lambda}  := 
\left\|(L_K+\lambda
I)^{-1/2}(L_{K,D}f^*-S_D^Ty_D)\right\|_K.
\end{eqnarray}
By the aid of the restriction of  $\varepsilon_i$, i.e., $E[\varepsilon_i]=0$, $L_{K,D}f^*$ can be regarded as a population of $S_D^Ty_D$ in terms that
$$
    L_{K,D}f^*-S_D^Ty_D=\frac1m\sum_{i=1}^m\varepsilon_iK_{x_i}.
$$
Therefore, $\mathcal P_{D,\lambda}$ can be easily derived from some Hilbert-valued concentration inequality \cite{Caponnetto2007}. In the following theorem, we present an error estimate for kernel interpolation with noisy data. 

\begin{theorem}\label{Theorem:Stability-via-operator}
	Let $0\leq u\leq 1/2$. If  $D=\{(x_i,y_i)\}_{i=1}^m$ satisfies  (\ref{model-2}) and $f^*$ satisfies (\ref{regularity-assump}) with $r\geq 1/2$, then
	\begin{eqnarray}\label{noise-operator}
	\|L_K^{u}(f_D-f^*)\|_K
	&\leq&
	\min_{\mu>0}(2\mu^{u-1/2}\mathcal Q^2_{D,\mu}+ {\mu}^{u+1/2} (\sigma^D_m)^{-1}\mathcal Q^{2u+1}_{D,\mu})  \mathcal P_{D,\mu} \nonumber\\
	&+&   \min_{\lambda>0} \left\{\begin{array}{ll}
	\lambda^{r+u-1/2}\mathcal Q_{D,\lambda}^{2r+2u-1}\|h^*\|_\rho, & \quad 1/2\leq r\leq 1,\\
	(r-1/2)\kappa^{r-3/2}\|h^*\|_\rho\lambda^u\mathcal Q_{D,\lambda}\mathcal R_D,  & \quad r>1.
	\end{array}
	\right.
	\end{eqnarray}
\end{theorem}

Compared with Theorem \ref{Theorem:noise-less-operator}, there is an additional term involving $\mathcal P_{D,\mu}$ and $\sigma_m^D$ in the right-hand side of (\ref{noise-operator}) to capture the stability of the kernel interpolation.  It should be highlighted that $\mathcal P_{D,\mu}$ reflects the quality of noise and $(\sigma_m^D)^{-1}$, as well as the condition number $\sigma_1^D/\sigma_m^D$, is a standard measurement for the stability of kernel interpolation \cite{Schaback1995}. Different from existing analysis of kernel interpolation for low dimensional data \cite{Schaback1995,Schaback2000,Wendland2004}, our result presented in Theorem \ref{Theorem:Stability-via-operator} is available to data sampled in arbitrary input space in the sense that there is not the widely used mesh norm $h_\Xi$ involved in our analysis. In particular,  as Figure \ref{fig:motivation_example} (b) purports to show,  the minimum eigenvalue (condition number) is extremely small (large)  for low-dimensional data, and will be relatively large (small) for high-dimensional data. This makes our result be more suitable for tackling high dimensional data and shows the reason why
 the regularization term in \cite{Caponnetto2007,Smale2007,Lin2017,hesse2017radial}  is removable for kernel interpolation. Based on Theorem \ref{Theorem:Stability-via-operator}   with $r=1$, $u=1/2$ or  $u=0$, we can get the following error estimate for kernel interpolation.

\begin{corollary}\label{Corollary:Stability-via-operator-rho}
If  $D=\{(x_i,y_i)\}_{i=1}^m$ satisfies  (\ref{model-2}) and $f^*$ satisfies (\ref{regularity-assump}) with $r=1$, then
\begin{eqnarray*}
	\|f_D-f^*\|_\rho
	\leq
	\min_{\mu>0}(2+ {\mu}(\sigma^D_m)^{-1}) \mathcal Q^2_{D,\mu} \mathcal P_{D,\mu} 
	+   \min_{\lambda>0} \lambda\mathcal Q^2_{D,\lambda}\|h^*\|_\rho 
	\end{eqnarray*}
and
\begin{eqnarray*}
	\|f_D-f^*\|_\infty
	\leq
	\min_{\mu>0}(2\mu^{-1/2}\mathcal Q^2_{D,\mu}+ {\mu}^{1/2}(\sigma^D_m)^{-1}\mathcal Q_{D,\mu})\kappa \mathcal P_{D,\mu} 
	+   \min_{\lambda>0} \kappa \lambda^{1/2}\mathcal Q_{D,\lambda}\|h^*\|_\rho.
	\end{eqnarray*}
\end{corollary}

As $\mathcal Q_{D,\lambda}$ and $\mathcal P_{D,\mu}$ decreases with respect to $\lambda$ and $\mu$ respectively, the minimization  concerning $\mu$ and $\lambda$ is well defined.   To derive the approximation error in detail, we need to tightly bound $\mathcal Q_{D,\lambda}$ and $\mathcal P_{D,\mu}$ via  the spectrum of kernel matrix and then derive verifiable error estimates for kernel interpolation, which is the aim of Section \ref{Sec.Random}.

\subsection{Kernel interpolation for trans-native space data}
If $f^*\in\mathcal H_K$, then the underlying kernel interpolation can be regarded as a projection from $\mathcal H_K$ to $\mathcal H_{K,m}:=\mbox{span}\{\phi_1^{D},\dots,\phi_m^{D}\}$, which makes the analysis expedient. This effective technique is lost when we face a target function $f^*\notin \mathcal H_K$. The ensuing difficulty is referred to as the  ``native space barrier'' in \cite{Narcowich2007}.  To overcome it, Narcowich et al.
made good use of estimates for the minimal eigenvalue of the kernel matrix $\mathbb K$ (in terms of the minimal separation of data sites). But this approach runs into obstacles when dealing with high-dimensional  data.  In this subsection, we conduct analysis by modifying the integral operator approach used in the previous two subsections.  For this purpose,
define $\mathcal S_{D}:L_{\rho_X}^2\rightarrow\mathbb R^{m}$  by  \footnote{In this paper, we use  notations $S_D$ and $\mathcal S_D$ to denote respectively  sampling operators on $\mathcal H_K$ and $L_{\rho_X}^2$. The action of $\mathcal S_D$ is restricted  to the totality of all continuous functions on $\mathcal{X}$. We note that $S_D$ is a continuous linear operator but $\mathcal S_D$ is not. }
$$
\mathcal S_{D}f:=(f(x_i))_{i=1}^m.
$$
Define further
$$
\mathcal L_{K,D}f:=  S_D^T\mathcal S_Df=\frac1{m}\sum_{i=1}^mf(x_i)K_{x_i},\qquad f\in L_{\rho_X}^2.
$$
For $f\in\mathcal H_K$, we have $\mathcal L_{K,D}f=L_{K,D}f$.
Under the noiseless setting (\ref{model-1}) with $f^*$ satisfying (\ref{regularity-assump}) for $0<r<1/2$,
it  follows from (\ref{interpolation-est}) that
\begin{equation}\label{operator-rep12}
f_D=  S_D^T(  S_D  S_D^T)^{-1} \mathcal S_Df^*.
\end{equation}

Different from (\ref{Theorem:noise-less-operator}), $f\notin\mathcal H_K$ requires novel technique in analysis, even for noise-free data.
The basic idea is to find a $f_\mathcal H\in\mathcal H_K$ such that $f_\mathcal H$ is a good approximation of $f^*$ and then regard (\ref{model-1}) as a noisy model with $f^*-f_\mathcal H$  the noise and $f_\mathcal H$ the target function. In this way, the stability of kernel interpolation  involving $(\sigma_m^D)^{-1}$ should be considered in the estimate.
\begin{eqnarray}
\mathcal W_{D,\lambda} &:=& \left\|(L_K+\lambda
I)^{-1/2}(L_{K,D}-L_K)\right\|,\\  \label{Def.W}
\mathcal U_{D,\lambda,g} &=&
\left\|(L_K+\lambda
I)^{-1/2}(\mathcal L_{K,D}g-\mathcal L_Kg)\right\|_K \label{Def.U}.
\end{eqnarray}
Similar as $\mathcal Q_{D,\lambda}$ and $\mathcal R_D$, $\mathcal W_{D,\lambda}$ is also a quantity to measure the difference  between
$L_{K,D}$ and $L_K$. 
Like
$\mathcal P_{D,\lambda}$, $\mathcal U_{D,\lambda,g}$ 
with some $g$ describes the effect of ``noise'', i.e. $f^*-f_\mathcal H$.
Since $f^*\not\in\mathcal H_K$, it is impossible to derive error estimate under the $\mathcal H_K$ norm like Theorem \ref{Theorem:noise-less-operator}. We conduct our analysis under the $L_{\rho_X}^2$ norm in the following theorem.

\begin{theorem}\label{Theorem:native-barrier-without-noise}
	If (\ref{model-1}) holds with a continuous $f^*$ satisfying (\ref{regularity-assump}) for $0\leq r<1/2$, then
	\begin{eqnarray}\label{Native-barrier-error}
	\|f^*-f_D\|_\rho&\leq& \min_{\lambda>0}\big\{
	(1+(1+\lambda(\sigma_m^D)^{-1})\mathcal Q_{D,\lambda}^{2r+2})  \lambda^r\|h^*\|_\rho\nonumber\\
	&+&
	(2+\lambda(\sigma_m^D)^{-1})\mathcal Q^2_{D,\lambda}\mathcal U_{D,\lambda,f^*}+
	\lambda^{r-1/2} \mathcal Q_{D,\lambda}^2\mathcal W_{D,\lambda}\|h^*\|_\rho\big\}.
	\end{eqnarray}
\end{theorem}

Similar as Theorem \ref{Theorem:Stability-via-operator}, the estimate in Theorem \ref{Theorem:native-barrier-without-noise} requires the stability of kernel interpolation, even though there is not any noise involved in the model (\ref{model-1}). As described above, for $f^*\notin\mathcal H_K$, we regard $f^*-f_{\mathcal H}$ as the noise in the analysis. Such a technique is not novel and has been adopted in \cite{Narcowich2007} for the same purpose for kernel interpolation on the sphere.  As a result, our result requires the minimal eigenvalue $\sigma_m^D$ of the kernel matrix to be not so small, which makes the analysis be only available for high-dimensional data.

We then show the roughness of $f$ satisfying (\ref{regularity-assump}) with $0\leq r<1/2$ by taking the algebraic decaying eigenvalues for examples. 
Assume $\mathcal X=\mathbb S^d$, the $d+1$-dimensional unit sphere embedded into $\mathbb R^{d+1}$ and
$\sigma_\ell\leq a\ell^{2\alpha}$ with $\alpha>1/2$ and $a>0$.
Then it is easy to check that $\mathcal H_K$ is equivalent to the Sobolev space   $\mathcal W_2^\tau(\mathbb S^d)$ with $\tau=\alpha d$ \cite{Narcowich2007}. As shown in \eqref{Regularity-1111}, different $r$   implies different eigenvalue decaying of $f^*$. 
In this way, (\ref{regularity-assump}) with $r$ running out $(0,1/2)$  is equivalent to interpolation space between $L^2_{\rho_X}$ and $\mathcal H_K$ with all indexes.

\section{Error Analysis via Spectrum of Kernel Matrix for Random Sampling}\label{Sec.Random}
Many existing error estimates for kernel interpolation are given in terms of $h_\Xi$, the mesh norm of an underlying sampling set $\Xi$; see e.g. \cite{Narcowich2002,Narcowich2004,Wendland2004,Narcowich2006,Narcowich2007}). Generally speaking, a smaller $h_\Xi$ gives rise to a more favorable error estimate. Assume quasi-uniformity of point distribution. Then we have $h_\xi\sim m^{-1/d}$, where $m$ is the cardinality of $\Xi$, which becomes infeasible when  $d$ is sufficiently large. More importantly, Monte Carlo simulations we have run based on  many high dimensional problems show that 
the mesh norm $h_\Xi$ and the minimal separation $q_{_\Xi}$ are lager than one with high probability. This feature has rendered many error-analysis techniques developed in the literature  ineffective in dealing with high dimensional problems. 
As mentioned in the introduction, we approach error analysis for kernel interpolation by adopting different measurements of operator differences 
  in which the spectrum of the kernel matrix plays a prominent role.  We point out that the error estimates presented here are probabilistic, which is the next best thing under situations  
where deterministic error analysis methods are impossible to implement. 
Motivations of our pursuit stem from two sources: (i) our simulations on condition numbers of Gaussian-kernel matrices (see Figure 2); (ii) the results we have gathered in Appendix B in which relations between finite-differences of operators and the spectrum of kernel matrices are given and in Appendix C which summarizes probabilistic estimates for lower bounds of $q_{_\Xi}$. As such, readers may find it helpful to review the pertinent results before proceeding. To be detailed, we carry out probabilistic error analysis   of kernel interpolation in terms of random sampling setting, i.e.  $\Xi=\{x_i\}_{i=1}^m$ are drawn i.i.d.  according to $\rho_X$. Such a setting has been adopted in the well known learning theory framework \cite{Smale2005,Smale2007,Caponnetto2007} and is   more suitable than the deterministic sampling setting for high-dimensional data \cite{hall2005geometric}.  

As the differences between $L_K$ and its empirical counterpart is not intuitive, we quantify the error analysis in terms of the spectrum of kernel matrix by using the recently developed integral operator technique in \cite{Lin2017,Guo2017,Blanchard2016}. For the sake of brevity, we only present   error estimates under the $L_{\rho_X}^2$ norm and those under the $\mathcal H_K$ norm as well as uniform norm can be derived by using results in the previous section and the same method in this section. 

Denote the eigenvalues of $\mathbb K$  by
$\{\sigma_{\ell,\mathbb K}\}_{\ell=1}^m$ with $ \sigma_{1,\mathbb K}\geq
\sigma_{2,\mathbb K} \ge \dots \geq \sigma_{m,\mathbb K}>0$.
From (\ref{operator-matrix}), we   have
\begin{equation}\label{operator-matrix-eigenvalue}
\sigma_{\ell,\mathbb K}=m\sigma^D_\ell,\qquad \ell=1,\dots,m.
\end{equation}
Write
\begin{equation}\label{Def.A}
\mathcal A_{D,\lambda}:=
\left(\frac{1}{m\lambda}+\frac{1}{\sqrt{m\lambda}}\right)\max\left\{1,\sqrt{\sum_{\ell=1}^m
	\frac{\sigma_{\ell,\mathbb K}}{\lambda m+\sigma_{\ell,\mathbb K}}}\right\}.
\end{equation}
It is easy to check that $\mathcal A_{D,\lambda}$  is non-increasing  with respect to the eigenvalues of the kernel matrix and decreases with respect to $\lambda$.    
The following theorem quantifies the error of kernel interpolation for noiseless data.

\begin{theorem}\label{Theorem:noise-less-eigenvalue}
	Let $0<\delta<1$. If $D=\{(x_i,y_i)\}_{i=1}^m$ satisfies (\ref{model-1}), $f^*$ satisfies (\ref{regularity-assump}) with $r\geq 1/2$ and $X=\{x_i\}_{i=1}^m$ are  drawn identically and independently according to $\rho_X$, then with confidence at least $1-\delta$, there holds
	\begin{equation}\label{noise-less-eigenvalue}
	\|f_D-f^*\|_\rho
	\leq
	C_1\log^{6}\frac{16}\delta\left\{\begin{array}{lll}
	&\min\limits_{\lambda>0}\{\lambda^r
	(\mathcal A_{D,\lambda}+1)^{2r}\} , \quad & 1/2\leq r\leq 1,\\
	&\frac1{\sqrt{m}} \min\limits_{\lambda>0}\{\sqrt{\lambda}(\mathcal A_{D,\lambda}+1)\},\quad  & r>1,
	\end{array}
	\right.
	\end{equation}
	where $C_1$ is a constant depending only on $\kappa$, $\|h^*\|_\rho$ and $r$.
\end{theorem}

 It should be noted that there are two terms,  $\sqrt{\lambda}$ and $\sqrt{\lambda}\mathcal A_{D,\lambda}$,  involved in the right-hand side of (\ref{noise-less-eigenvalue}). The former  increases with respect to $\lambda$ while the latter decreases with $\lambda$.   Therefore, there is a  unique $\lambda_0$ minimizing $\sqrt{\lambda}+\sqrt{\lambda}\mathcal A_{D,\lambda}$, making our estimate be well defined.
In view of (\ref{Def.A}), Theorem \ref{Theorem:noise-less-eigenvalue} shows that the approximation error of kernel interpolation can be given in terms of the trace of the matrix $\mathbb K(\mathbb K+\lambda m I)^{-1}$, which depends only on the spectrum of the kernel matrix. From (\ref{noise-less-eigenvalue}), we conclude that kernel matrices with smaller eigenvalues perform better than those with larger ones. Thus, for interpolation of noise-free data with target functions from the native space, kernels with small eigenvalues, such as the Gaussian kernel, are preferable.  The main difference between Theorem \ref{Theorem:noise-less-eigenvalue} and existing results in \cite{Narcowich2002,Narcowich2004,Narcowich2006,Narcowich2007} is that there are not any measurements for the distributions of scattered data such as the mesh norm and separation radius that would be large when $d$ is large involved in our error estimate of noise-free data.

Though kernel interpolation performs well for noise-free data, 
data obtained from modeling  real-world problems often contain noises. To tolerate noises,  the kernel matrix must be well-conditioned. In the following theorem, we quantify the approximation performance  of kernel interpolation with noisy data via the spectrum of kernel matrix.

\begin{theorem}\label{Theorem:Stability-via-eigenvalue}
	Let $0<\delta<1$. If $D=\{(x_i,y_i)\}_{i=1}^m$ satisfies (\ref{model-2}), $f^*$ satisfies (\ref{regularity-assump}) with $r\geq 1/2$ and $\Xi=\{x_i\}_{i=1}^m$ are i.i.d. drawn according to $\rho_X$, then with confidence at least $1-\delta$, there holds
	\begin{eqnarray}\label{noise-eigenvalue}
	 \|f_D-f^*\|_\rho
	&\leq&
	C_2\min_{\mu>0}\left\{(1+ {\mu}m (\sigma_{m,\mathbb K})^{-1}) ( \mathcal A_{D,\mu}+1)^2
	\sqrt{\mu}\mathcal A_{D,\mu}\right\} \log^6\frac{16}{\delta}\nonumber\\
	&+&    C_1\log^{6}\frac{16}\delta\left\{\begin{array}{ll}
	\min\limits_{\lambda>0}\{\lambda^r
	(\mathcal A_{D,\lambda}+1)^{2r}\} , & \quad 1/2\leq r\leq 1,\\
	\frac1{\sqrt{m}} \min\limits_{\lambda>0}\{\sqrt{\lambda}(\mathcal A_{D,\lambda}+1)\}, & \quad r>1,
	\end{array}
	\right.
	\end{eqnarray}
	where
	$C_2$ is a constant depending only on $\kappa$, $\gamma$ and $\|f^*\|_\infty$.
\end{theorem}

It is noteworthy to mention that (\ref{model-2}) can tolerate a noise level comparable to the magnitude of  $y_D$.  Therefore, besides the  noise-free approximation error (the second term in the right-hand side of (\ref{noise-eigenvalue})), it requires an additional term involving the smallest  eigenvalue of $\mathbb K$ to reflect the stability of kernel interpolation. Such an additional term presents a restriction on the selection of the kernel of the input dimension $d$. For noise-free data,  Theorem \ref{Theorem:Stability-via-eigenvalue} illustrates the the eigenvalues of the kernel matrix the smaller the better, implying both small $d$ and kernel with fast eigenvalue decaying. This situation changes dramatically when the data are noised. In particular, as shown in the next section or Appendix C, the smallest eigenvalue of kernel matrix increases with respect to $d$. This makes our analysis be available for data sampled in a high dimensional input space.

Since analysis of approximating noisy data has been conducted in \cite{Smale2007,Caponnetto2007,SteinwartHS,Lin2017} by using the well known kernel ridge regression (KRR) algorithm:
\begin{equation}\label{KRR}
    f_{D,\lambda}=\arg\min_{f\in\mathcal H_K}\frac1m\sum_{i=1}^m(f(x_i)-y_i)^2+\lambda\|f\|_K^2, \qquad \lambda>0.
\end{equation}
The main reason of introducing the regularization term is to overcome well known over-fitting phenomenon and enhance the stability of the approximant. Our algorithm can be regarded as a kernel ``ridgeless'' regression via taking $\lambda=0$. Theorem \ref{Theorem:Stability-via-eigenvalue} together with eigenvalue estimate in the next section and Appendix C then implies that if $d$ is sufficiently large, there is not over-fitting for kernel interpolation.

Finally, we   quantify the approximation error  for kernel interpolation of trans-native space data in terms of the spectrum of $\mathbb K$ in the following theorem.

\begin{theorem}\label{Theorem:native-barrier-without-noise-eigenvalue}
	Let $0<\delta<1$. Suppose that $D=\{(x_i,y_i)\}_{i=1}^m$ satisfies (\ref{model-1}), that $f^*$ is continuous and satisfies (\ref{regularity-assump}) with $0<r<1/2$, and that  $\Xi=\{x_i\}_{i=1}^m$ are  drawn i.i.d. according to $\rho_X$. Then with confidence at least $1-\delta$, there holds
	\begin{eqnarray}\label{Native-barrier-error-eigenvalue}
	  \|f^*-f_D\|_\rho &\leq& C_3\log^{6}\frac{24}\delta
	\min_{\lambda>0}\big\{
	(1+(1+\lambda m(\sigma_{m,\mathbb K})^{-1})
	\mathcal A_{D,\lambda}+1)^{2r+2}    \lambda^r\nonumber\\
	&+&
	((1+\lambda m(\sigma_{m,\mathbb K})^{-1}) +
	\lambda^{r-1/2}) ( \mathcal A_{D,\lambda}+1)^{2}\sqrt{\lambda}\mathcal A_{D,\lambda}  \big\},
	\end{eqnarray}
	where $C_3$ is a constant depending only on $\kappa$, $\|h\|_\rho$, $\|f^*\|_\infty$ and $r$.
\end{theorem}

By adeptly coordinating  and manipulating decay rates and bandwidths of Fourier transforms,  Narcowich et al \cite{Narcowich2004, Narcowich2006} found a marvelous way of projecting the approximation power of a higher-order Sobolev spline kernel into a larger RKHS associated with a lower-order Sobolev spline kernel; see also \cite{levesley-sun}. This method  depends in a crucial way on the quasi-uniformity of sampling-point distribution. In contrast, Theorem \ref{Theorem:native-barrier-without-noise-eigenvalue} only requires the presence of the spectrum of the underlying kernel matrix to achieve the desired stochastic approximation goal, the passage of which is reflected in the appearance of the extra quantity $\lambda m(\sigma_{m,\mathbb K})^{-1}$ in the error estimate. This is noticeably different from the case $f^*\in\mathcal H_K$.
	
The estimate given in Theorem \ref{Theorem:native-barrier-without-noise-eigenvalue} strongly indicates the importance of a well-conditioned  kernel matrix in overcoming ``the native space barrier". In this way, our analysis is only suitable for  high dimensional data.

\section{Spectrum  Analysis for Random Kernel Matrices}\label{sec.spectrum}

In this section, we assume that
$\Xi$ is generated by $m$-independent copies of the uniformly distributed random variable on $\mathcal X,$ and estimate the spectra of the ensuing random kernel matrices. Spectral analysis for other probabilistic distribution and dot product kernels can be found in \cite{Elkaroui2010,Liang2020}.
Throughout this section, we work with radial kernels. 

Let $(\sigma_\ell,\phi_\ell/\sqrt{\sigma_\ell})$ be the eigen-pairs of the integral operator $\mathcal L_{K}$ (defined in \eqref{integral-1}). For the special case in which  $\mathcal X$ is the unit (open) ball of $\R^d$, Steinwart at al \cite{SteinwartHS} showed that eigenvalues $\sigma_\ell$ associated with the reproducing kernel of the Sobolev space $W^{\tau}_2(\mathcal X)$ $(\tau > d/2)$ satisfy $\sigma_\ell \leq c_0 \ell^{-2 \tau/d}, \; \ell \in \Z_+$, where $c_0$ is an absolute constant. (These are also referred to in the literature as Sobolev spline kernel.) That is, $\sigma_\ell$ satisfy inequality \eqref{Eigenvalue-decay-ass-alg} below.
Inspecting pertinent work in \cite{Blanchard2012} and \cite{Belkin2018}, one concludes that eigen-values $\sigma_\ell$ associated with Gaussian kernels satisfy inequality \eqref{Eigenvalue-decay-ass-exp} below.  Accordingly, we assume in the sequel that the eigen-value sequence of the integral operator $\mathcal L_{K}$ (defined in (\ref{integral-1})) satisfies one of the following two inequalities:
\begin{align}
\sigma_\ell  \leq & c_0\  \ell^{-\beta},\quad \quad \beta>1; \label{Eigenvalue-decay-ass-alg}\\
\sigma_\ell \leq & c_0\  e^{-\alpha\ \ell^{1/d}}, \;\ \alpha>0,\label{Eigenvalue-decay-ass-exp}
\end{align}
in which $c_0>0$ is an absolute constant. 
The following proposition gives an upper-bound estimate of $\mathcal A_{D,\lambda}$ under the above conventions.

\begin{proposition}\label{Proposition:uniform-spectrum}
	Let $0<\delta<1$ be given. Then for any $0 < \lambda \le 1$, the following inequalities hold true with confidence $1-\delta$,
	\begin{eqnarray*}
	\mathcal A_{D,\lambda}\leq  C_4 \left(\frac{1}{m\lambda}+\frac{1}{\sqrt{m\lambda}}\right)\left(1+\frac{1}{m\lambda}\right)\log^2\frac{4}{\delta}
	\left\{\begin{array}{ll}
	\lambda^{-1/(2\beta)},  & \quad \mbox{if (\ref{Eigenvalue-decay-ass-alg})},\\
	\sqrt{d!}\ \alpha^{-d/2}\  \log^{d/2}\frac1\lambda, & \quad \mbox{if (\ref{Eigenvalue-decay-ass-exp})},
	\end{array}
	\right.
	\end{eqnarray*}
	where $C_4$ is a constant depending only on $c_0$.
\end{proposition}

We devote the rest of the section to lower bounds of the minimal eigenvalue of the kernel matrix $\mathbb K$, which has been studied extensively in the radial basis function research community; see \cite{Narcowich1991,Ball1992,Schaback1995,Wendland2004} and the references therein. The main theme of the research is to bound 
the smallest eigenvalue of $\mathbb K$ in terms of  the separation radius. For Gaussian kernel $G_a$ and Sobolev spline kernel
$S_\tau$ defined respectively by:
\begin{eqnarray*}
G_a(x,x')&=&e^{-a\|x-x'\|_2^2}\; ( a>0),  \\
S_\tau(x,x')&=&\frac{2\pi^d}{\Gamma(\tau)}\mathbb B_{\tau-d/2}(\|x-x'\|_2)(\|x-x'\|_2/2)^{\tau-d/2} \; (\tau>d/2),
\end{eqnarray*}
where $ \mathbb B_{\nu}(t)$ is the modified Bessel function of the second kind, we find in \cite[Table 12.1]{Wendland2004} the following estimates:
\begin{eqnarray}
\sigma_{m,G_a}&\geq&
\frac1{2^{2d+1}\Gamma(d/2+1)}\left(\frac{6.38d}{q_{_\Xi}\sqrt{a}}\right)^d
\exp\left[-\left(\frac{6.38d}{q_{_\Xi}\sqrt{a}}\right)\right],\label{stability-for-Gaussian}\\
\sigma_{m,S_\tau}  &\geq&
\frac{q_{_\Xi}^{2\tau-d}}{2^{2\tau+2d+1}\pi^{d/2}\Gamma(d/2+1)}
\frac1{(6.38d)^{2\tau-d}}\left(1+\frac{q_{_\Xi}^2}{162.8d^2}\right)^{-\tau}.\label{stability-for-Sobolev-111}
\end{eqnarray}
Making use of the two inequalities above, Lemma \ref{Lemma:mesh-separation} in Appendix C (or Lemma \ref{Lemma:separation-for-gaussian} in Appendix C for the normal distribution), we derive estimates for $\sigma_{m,G_a} $ and $\sigma_{m,S_\tau}$. These results join forces  with Proposition \ref{Proposition:uniform-spectrum} and approximation results in the previous section, and give stochastic  error estimates for kernel interpolations with many highly-applicable kernels. We present here such error estimates for kernel interpolations while  Sobolev spline kernels and Gaussian kernels are employed. 
%
%
\begin{corollary}\label{Corollary:noise-less-rate}
	Let $0<\delta<1$. If $K(\cdot,\cdot)= S_\tau(\cdot,\cdot)$ with $\tau>d/2$,  $\mathcal X=[0,1]^d$, $D=\{(x_i,y_i)\}_{i=1}^m$ satisfies (\ref{model-1}), $f^*$ satisfies (\ref{regularity-assump}) with $r\geq 1/2$, $\Xi$ is generated by $m$ independent copies of the random variable uniformly distributed in $\mathcal X$.
	Then with confidence at least $1-\delta$, we have
	\begin{equation}\label{noise-less-rate}
	\|f_D-f^*\|_\rho
	\leq
	C_5\log^{8}\frac{16}\delta\left\{\begin{array}{ll}
	m^{-\frac{2r\tau}{2\tau+d}}, & 1/2\leq r\leq 1,\\
	m^{-\frac{ 2\tau+d/2}{2\tau+d}}, & r>1,
	\end{array}
	\right.
	\end{equation}
	where $C_5$ is a constant depending only on $C_1$, $C_4$ and $r$.
\end{corollary}

\begin{corollary}\label{Corollary:noise-less-rate-G}
	Let $0<\delta<1$. If $K(\cdot,\cdot)= G_a(\cdot,\cdot)$ with $a>0$,  $\mathcal X=[0,1]^d$, $D=\{(x_i,y_i)\}_{i=1}^m$ satisfies (\ref{model-1}), $f^*$ satisfies (\ref{regularity-assump}) with $r\geq 1/2$, $\Xi$ is generated by $m$ independent copies of the random variable uniformly distributed in $\mathcal X$.
	Then with confidence at least $1-\delta$, we have
	\begin{equation}\label{noise-less-rate-G}
	\|f_D-f^*\|_\rho
	\leq
	C_6\sqrt{d}a^{-d/2}\log^{8}\frac{16}\delta\left\{\begin{array}{ll}
	(m^{-1}\log^{d}m)^r, & 1/2\leq r\leq 1,\\
	m^{-1}\log^{d/2}m, & r>1,
	\end{array}
	\right.
	\end{equation}
	where   $C_6$ is a constant depending only on $C_1$, $C_4$ and $r$
\end{corollary}

\section{Numerical Results}\label{Sec.Experiment}

In this section,  we present results of both toy simulations and a real world data experiments to verify our theoretical assertions and show the performance of kernel interpolation in tackling high dimensional data. Given a randomly generated (according to uniform distributions as mentioned above) training set $D=\{(x_i,y_i)\}_{i=1}^m$, we construct an approximant $f_{D,\gamma}$ with kernel parameter $\gamma>0$ of the form:
\[
f_{D,\gamma}=\sum_{i=1}^m(\mathbb{G}^{-1}_\gamma y_D)_i G_\gamma(x_i,\cdot),
\]
in which $\mathbb{G}_\gamma$ denotes the corresponding kernel matrix, $(V)_i$ the $i$th component of the vector $v$. 

\subsection{Toy simulations}
This subsection conducts three  simulations to substantiate numerically our main theoretical findings. In the first simulation, we show that the spectrum of a kernel matrix $\mathbb{K}$ is a suitable barometer to gauge the behavior of $\|f_{D,\gamma} - f^*\|_\rho$ in high dimensional spaces, which offers strong numerical evidences to support the theoretical results. The second simulation aims to study the     role of dimensionality  in reflecting the approximation performance of interpolation, showing the necessity of high dimensionality for kernel interpolation. 
The last simulation is designed to be a comprehensive study of quasi-interpolation with different regularization parameters in high dimensional spaces to show the redundancy  of the regularization term for high dimensional data. 
We reiterate here  that quasi-interpolation reduces to interpolation if the regularization parameter is set to zero.

{\bf Simulation I.} In this simulation, we set $\mathcal X=[-1, 1]^d$, which we simply refer to as the cube in the sequel.  We generate $m\in\{500,600, \dots, 1500\}$ samples for training and the inputs 
$
\{x_i=(x_{i,1},x_{i,2},\dots,x_{i,d})^T\}_{i=1}^m
$
are independently drawn according to the uniform distribution on the cube. Then Gaussian kernel $G_\gamma(x, x')=\exp\left(-\frac{\gamma\|x - x'\|^2}{2}\right)$ with the kernel parameter $\gamma=0.05$ is employed for constructing the kernel matrix $\mathbb{K}$. We further define the main item of Theorem  \ref{Theorem:noise-less-eigenvalue} as the approximation error:
\begin{equation}
   \text{AE} =\left\{\begin{array}{ll}
	\min\limits_{\lambda>0}\{\lambda^r
	(\mathcal A_{D,\lambda}+1)^{2r}\} , & \quad 1/2\leq r\leq 1,\\
	\frac1{\sqrt{m}} \min\limits_{\lambda>0}\{\sqrt{\lambda}(\mathcal A_{D,\lambda}+1)\}, & \quad r>1.
	\end{array}
	\right.
\end{equation}

The corresponding outputs $\{y_i\}_{i=1}^m$ are generated from the following regression model:
\begin{equation}\label{simulation-noiseless-model}
y_i = f^*(x_i)  =\sum_{j=1}^d c_j\exp(-x_{i,j}^2),
\end{equation}
where the regression coefficients $(c_1,\dots,c_d)^T$ are sampled according to the uniform distribution on $[-1,1]^d$.
This means $f^*\in\mathcal{H}_K$, so we can set $r=1/2$ in our simulation study. We run 20 independent trials of the simulation and depict the mean values of AE and the rooted mean square error (RMSE) of 500 test samples in Figure \ref{fig:sim-1}.
\begin{figure}
	\centering
	\subfigure[]{\includegraphics[scale=0.33]{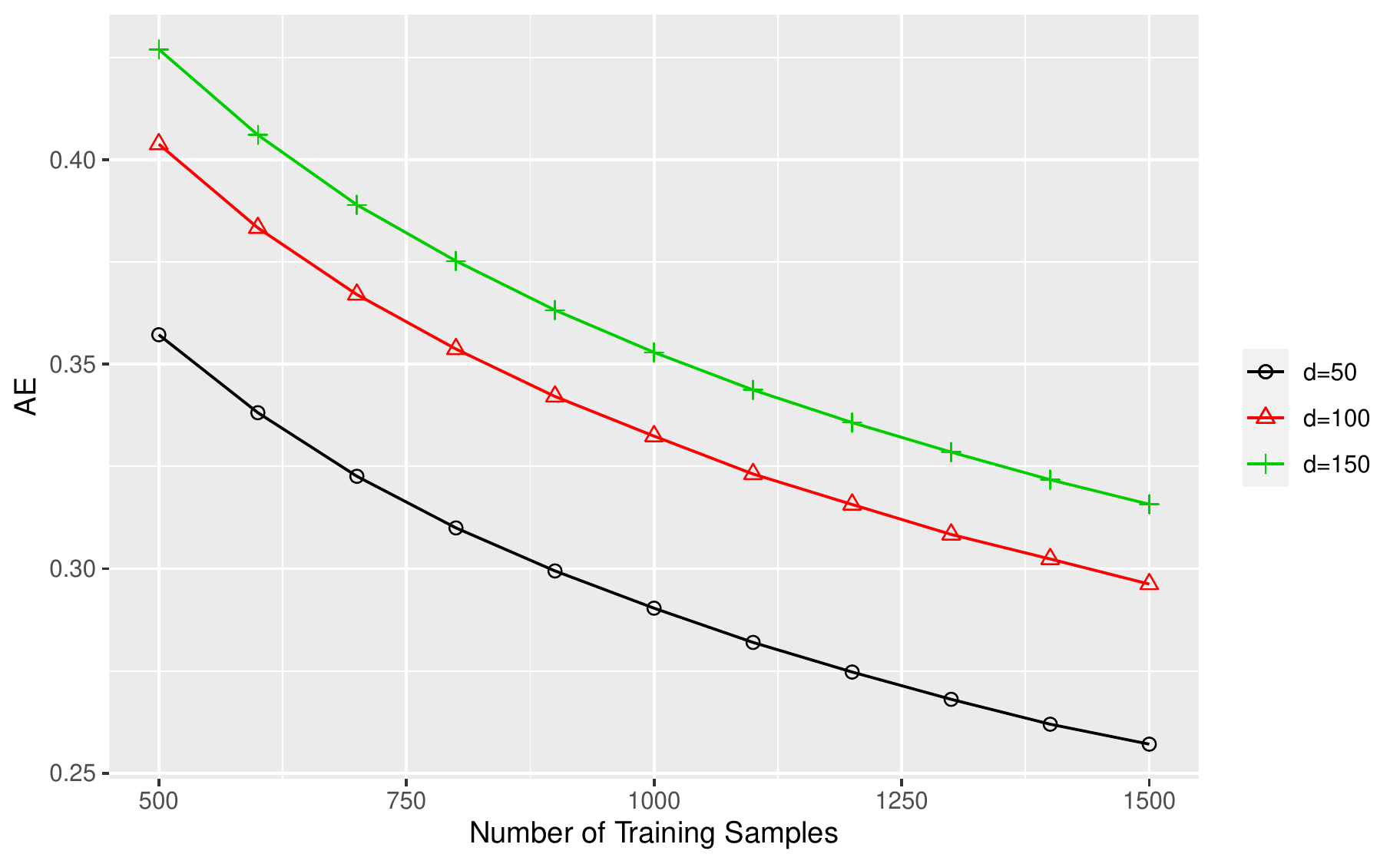}}
	\subfigure[]{\includegraphics[scale=0.33]{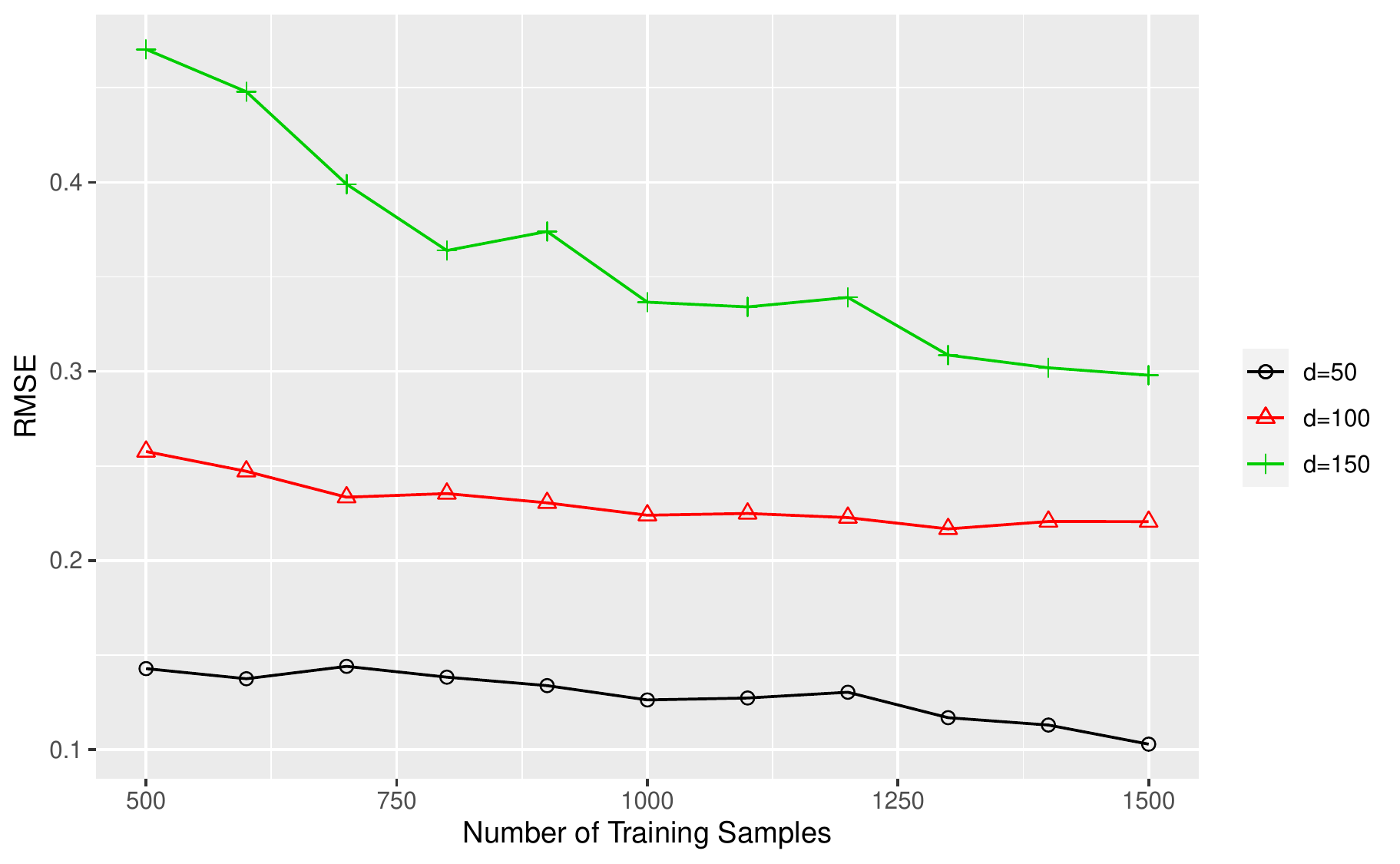}}
	\caption{Approximation error of kernel interpolation}
	\label{fig:sim-1}
\end{figure}

There are three important findings in Figure \ref{fig:sim-1}. At first, Figure  \ref{fig:sim-1} (b) exhibits a monotonously decreasing trends of RMSE with respect to the size of samples, which numerically verifies  Corollary \ref{Corollary:noise-less-rate-G}. Then,  Figure \ref{fig:sim-1} demonstrates that there is a close relation between RMSE of kernel interpolation and AE for kernel matrix in both their values and trends, showing that AE is an excellent upper bound of RMSE for kernel interpolation. This verifies Theorem \ref{Theorem:noise-less-eigenvalue}. Finally, although our derived approximation error is independent of the dimension $d$,   RMSE   behaves worse for larger $d$. The main reason is that we absorb
the dimension in the constant term, just as Corollary \ref{Corollary:noise-less-rate-G} purports to show.

{\bf Simulation II.} The simulation setting in this part is similar as that in Simulation I. The only difference is that we   set $\mathcal X=[-1, 1]^d$ with $d$ varying in $[50,100,\dots,500]$ and $m$ to be chosen from $\{300,600,900\}$. Before presenting our simulation, we at first show that the classical mesh norm (and separation radius) in \cite{Schaback1995,Schaback2000,Narcowich2002,Narcowich2004,Wendland2004,Narcowich2006,Narcowich2007} is unavailable to  gauge the approximation error in the high dimensional setting. In particular, it can be found in Figure \ref{fig:sim-2} that even for the separation radius (SR) increases with respect to $d$ and will larger than 1, provided $d$ is larger than a small value. As a result, the mesh norm is also larger than one and thus is not suitable to measure the approximation error by showing an approximation rate as $h_{\Xi}^\alpha$ for some $\alpha>0$. Differently, as shown in Figure \ref{fig:sim-2} (b),
the condition number as well as $(\sigma_{m,\mathbb K})^{-1}$ decreases with respect to $d$,  implying the well-conditioness of the kernel matrix for high-dimensional data. In this way, Figure \ref{fig:sim-2} verifies the feasibility of utilizing the spectrum of kernel matrix to take place of the mesh norm to gauge the approximation error for high dimensional data.
\begin{figure}
	\centering
	\subfigure[]{\includegraphics[scale=0.33]{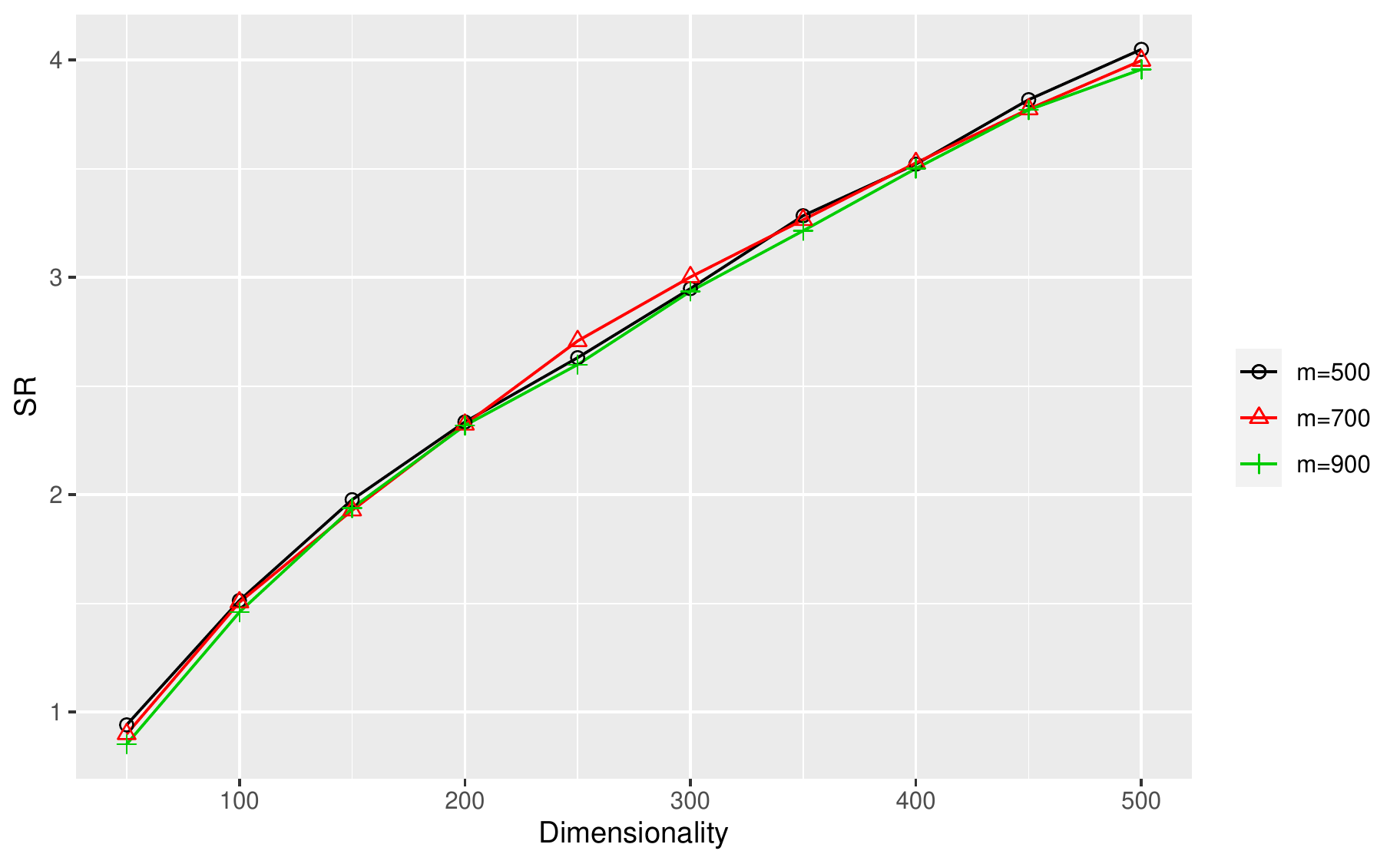}}
	\subfigure[]{\includegraphics[scale=0.30]{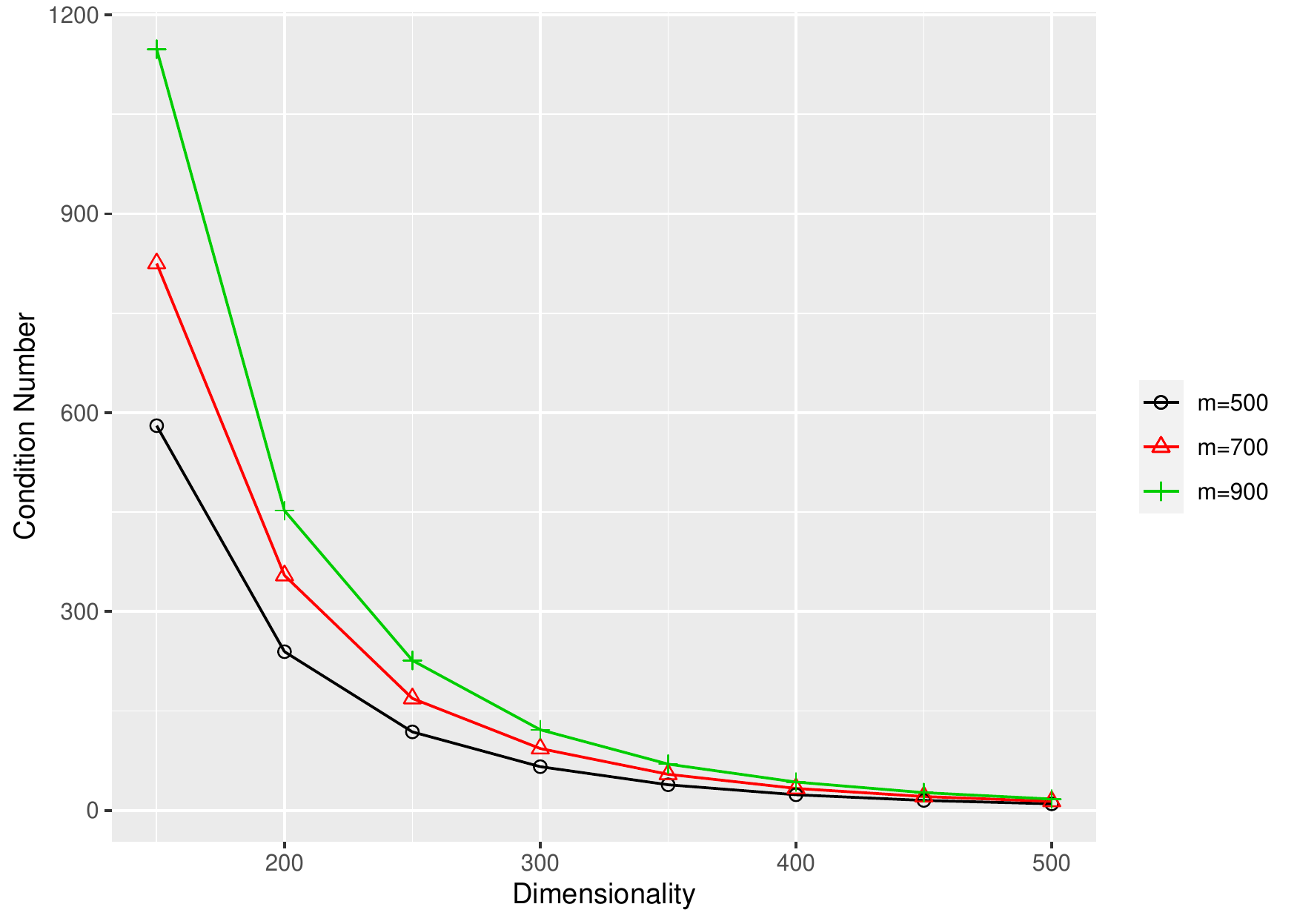}}
	\caption{Trends of separation radius and condition number of kernel matrix with respect to the dimension}
	\label{fig:sim-2}       
\end{figure}

Based on these observations, we then show the simulation results of kernel interpolation in terms of quantifying the relation among AE, RMSE and  the number of dimension. 
 The numerical results can be found in Figure \ref{fig:sim-3}. 

\begin{figure}
	\centering
	\subfigure[]{\includegraphics[scale=0.33]{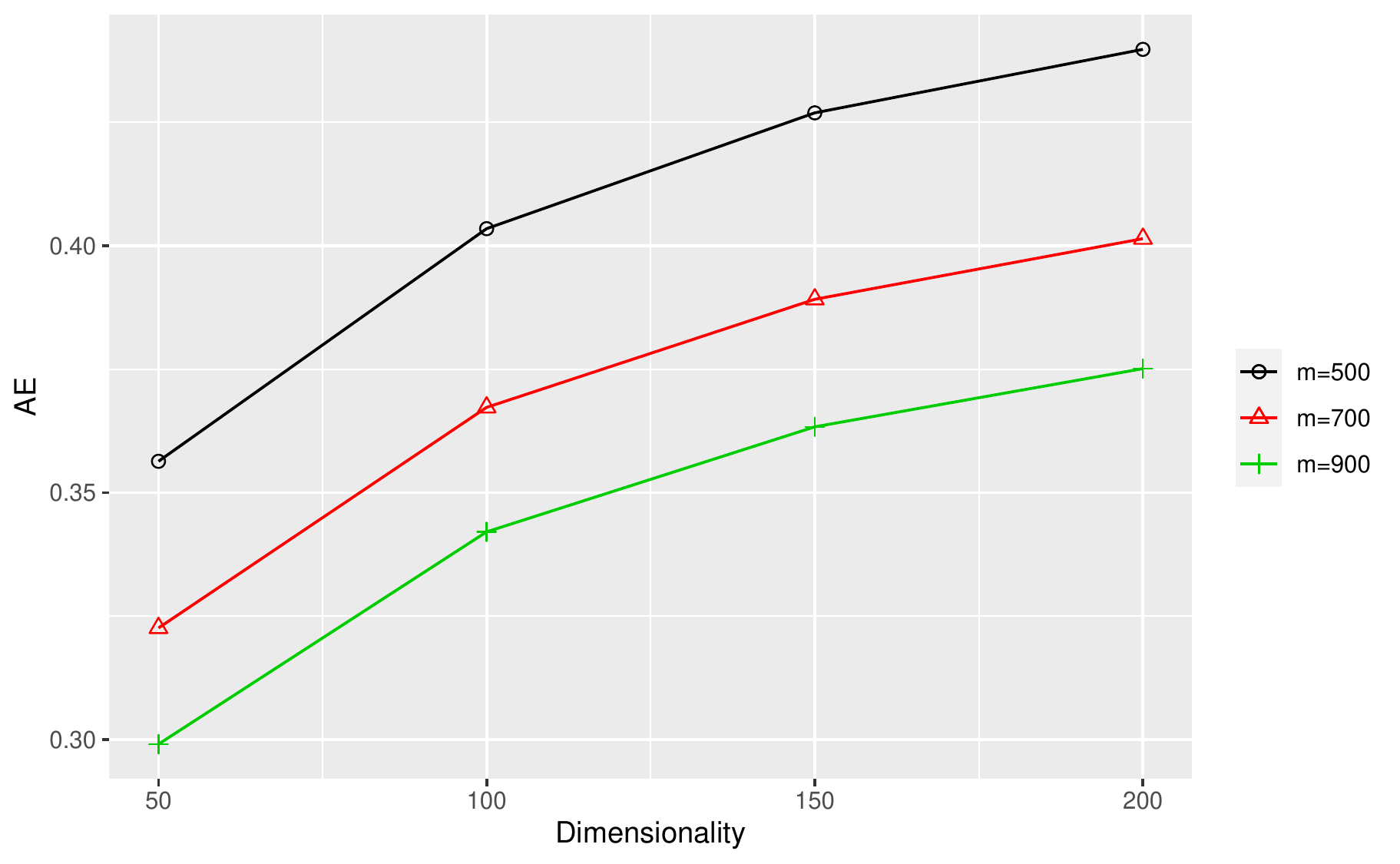}}
    \subfigure[]{\includegraphics[scale=0.33]{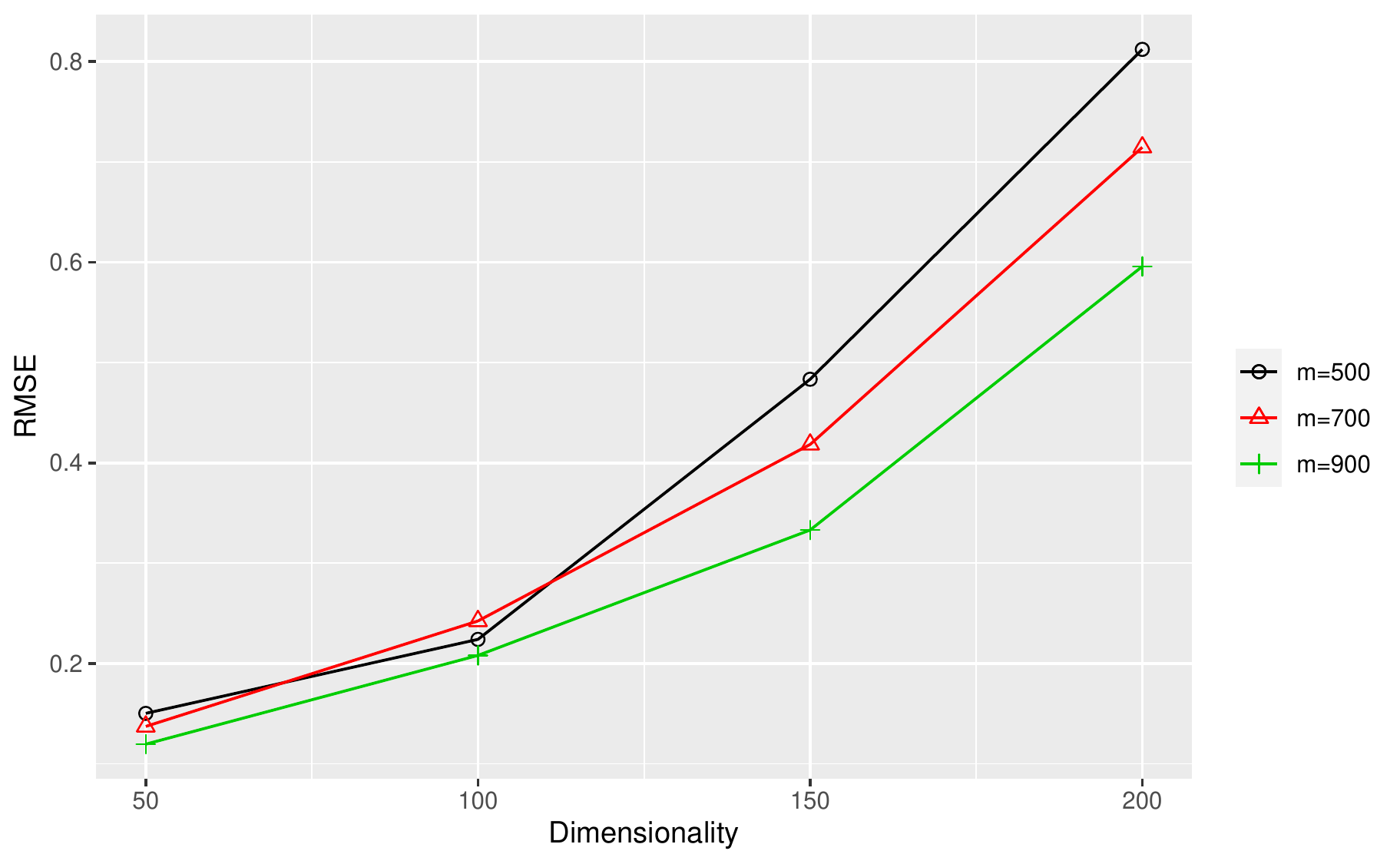}}
	\caption{Trends of AE  and RMSE  of  interpolation with respect to the dimension  }
	\label{fig:sim-3}
\end{figure} 

From Figure \ref{fig:sim-3}, we can also find three interesting phenomena. Firstly, Figure \ref{fig:sim-3} shows that both AE and RMSE decreases with respect to $d$, implying the difficulty for high dimensional interpolation. As shown in Corollary \ref{Corollary:noise-less-rate-G}, although the approximation rate does not heavily depend on $d$ (only a logarithmic relation), the approximation bound becomes worse when $d$ is larger. Secondly, As shown in Figure \ref{fig:sim-3} (b), RMSE increases  dramatically with respect to $d$, which   little bit contradicts with our theoretical assertions. We highlight that the main reason is that we employ a uniform kernel parameter $\gamma=0.05$ for all dimensions. Finally,   Figure \ref{fig:sim-3} also exhibits that AE is a feasible upper bound of RMSE, and thus verifies Theorem \ref{Theorem:noise-less-operator}.




{\bf Simulation III.}  In this simulation, we show the redundancy of the regularization term. For this purpose, we study both noise-free model (\ref{simulation-noiseless-model}) and the noisy model
\begin{equation}\label{simulation-noisy-model}
y_i = f^*(x_i) + \epsilon_i=\sum_{j=1}^d c_j\exp(-x_{i,j}^2) +\epsilon_i,
\end{equation}
where $\epsilon_1, \ldots, \epsilon_m$ are   sampled independently and identically   according to the uniform distribution on $[-0.2, 0.2]$.
We caution that the value of the tuning parameter $\gamma$ affects significantly the performance of $f_{D,\gamma}$. We experimented with several other ways, and eventually settled upon  the so-called ``hold-out method" \cite{Yao2007} in selecting a suitable value for $\gamma$. Roughly speaking,  the hold-out method divides the data set $D$ into training and validation set $D_{tr}$ and $D_{vl}$ respectively, where $D=D_{tr}\bigcup D_{vl}, D_{tr}\bigcap D_{vl}=\emptyset$ and  $D_{tr}$ contains half of the whole sample data. It then evaluates the performance of $f_{D,\gamma}$ for different values of $\gamma$ via the root mean square error (RMSE) on $D_{vl}$, and select the best value $\gamma^*$ by the following rule:
$$\gamma^*=\arg\min\left\{\sqrt{\frac{1}{2m}\sum_{\{x_i,y_i\}\in D_{vl}}(y_i - f_{D,\gamma}(x_i))^2}\right\}.$$
We then compute the RMSE of $f_{D,\gamma^*}$ against a randomly generated  testing data set $D_{test}=\{(x'_i,f^*(x'_i))\}_{i=1}^{500}$. To show the versatility  of our kernel interpolation method in high dimensional spaces, we generate some random  training samples $D_m:=\{(x_i,y_i)\}_{i=1}^m$ for $m=500,700, \dots, 1500,$ and testing samples $D_{test}=\{(x'_i,f^*(x'_i))\}_{i=1}^{500}$. We use a quasi-interpolation method (with a regularization parameter) to construct estimators of $f^*$ as follows. 
\begin{equation}\label{quasi-regularization}
f_{D,\gamma,\lambda}=\sum_{i=1}^m\left((\mathbb{G}_\gamma+\lambda I_m)^{-1} y_D\right)_i K_\gamma(x_i,\cdot),
\end{equation}
in which values of the regularization parameter $\lambda$ are respectively set to be 
\[
0, \quad 0.01,\quad 0.02,\quad 0.04,\quad 0.08\quad, 0.16,\quad 0.32,\quad 0.64,\quad 1.28,\quad 2.56.
\]
(When $\lambda=0$,  $f_{D,\gamma,0}$ is the  estimator kernel interpolation). We run 20 independent trials for each individual case. Average values of RMSE for different regularization parameters are shown in Part (b) of Figure \ref{fig:sim-4}. Some observations are in order.



\begin{figure}
	\centering
	\subfigure[]{\includegraphics[scale=0.33]{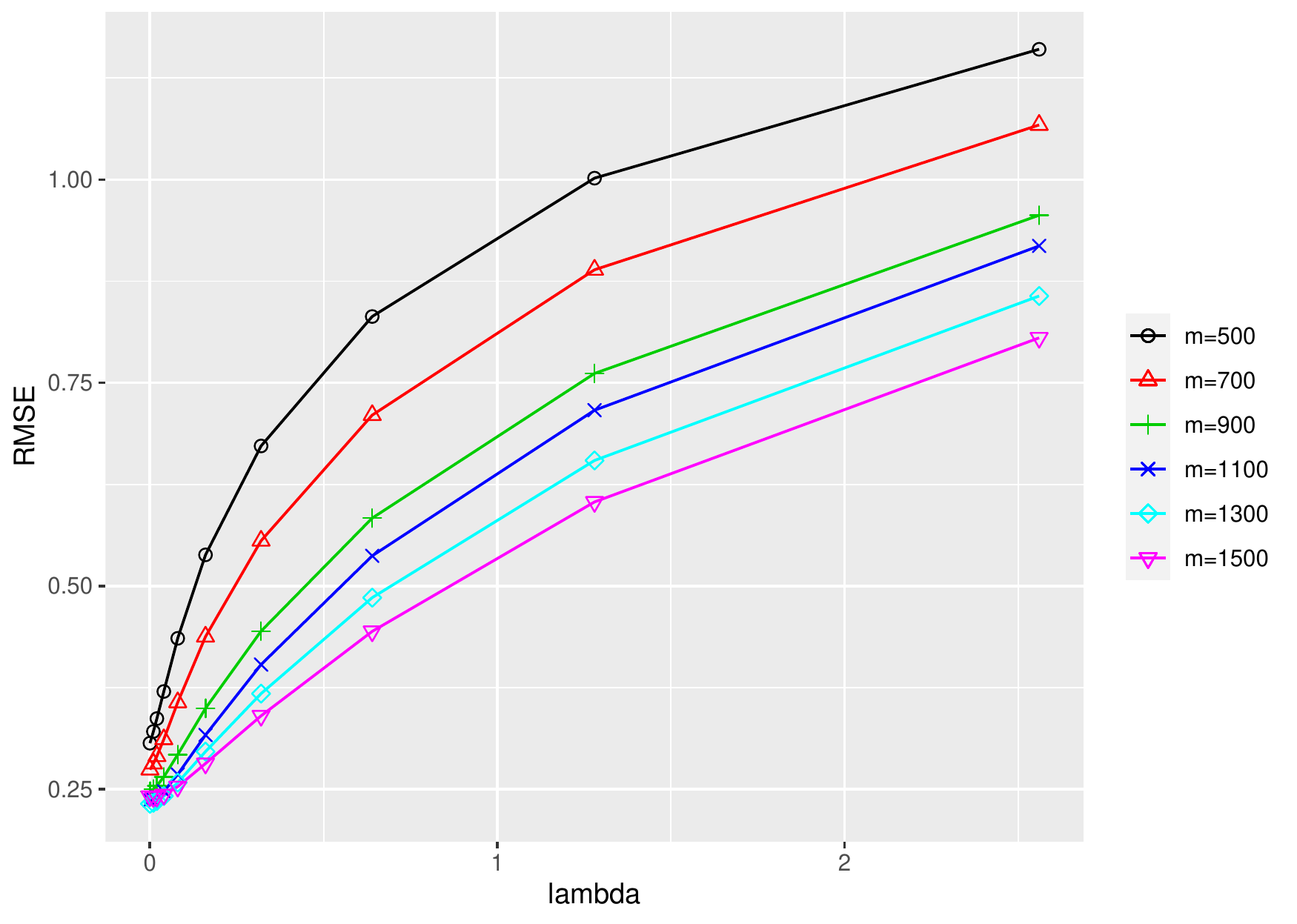}}
	\subfigure[]{\includegraphics[scale=0.33]{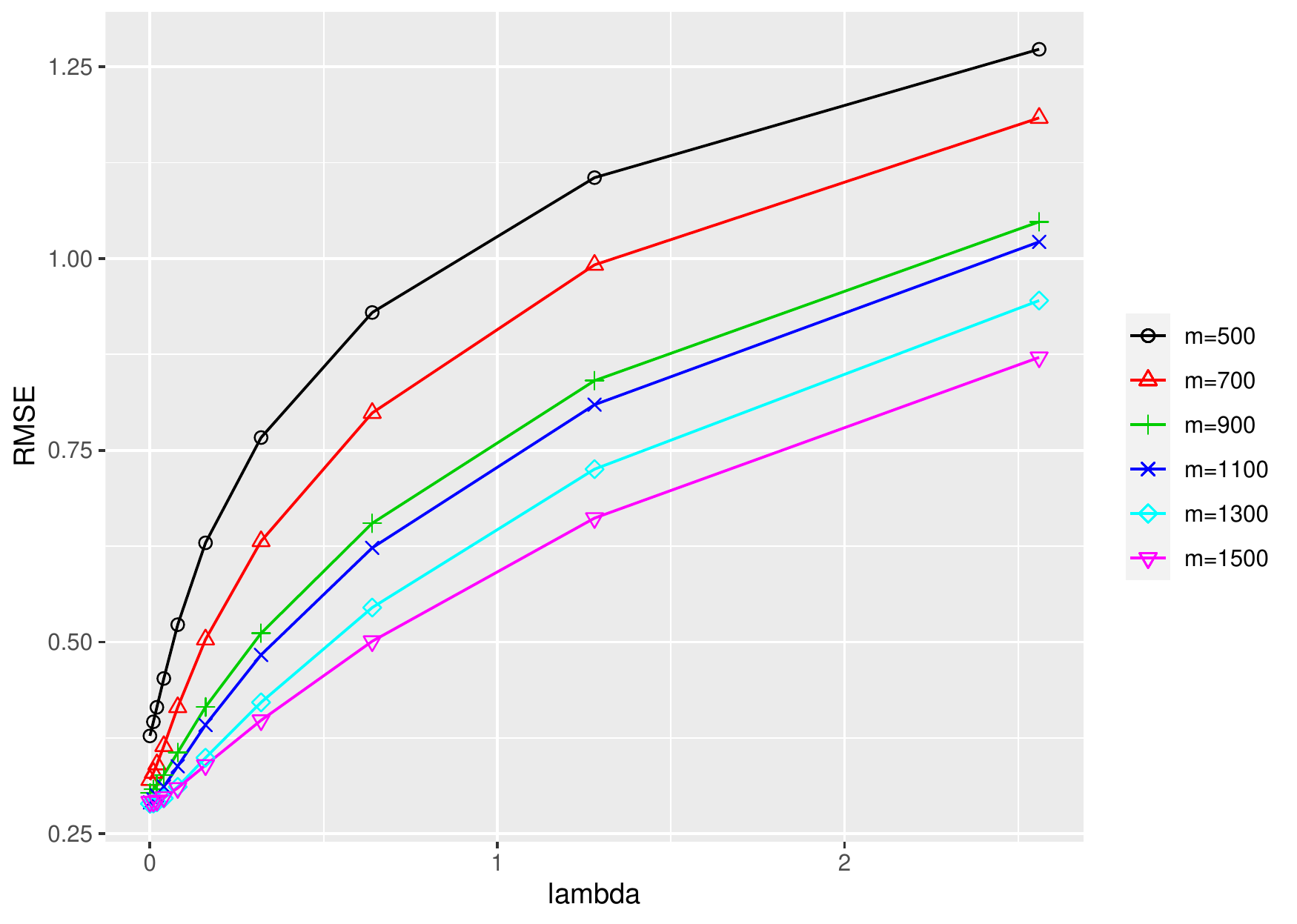}}
	\caption{RMSE for the noise-free and noise interpolation simulations}
	\label{fig:sim-4}
\end{figure}

Figure \ref{fig:sim-4} illustrates that the minimal RMSE of quasi-interpolants of various training sets and different regularization parameter ($\lambda$) values appears to reach at $\lambda=0$. This offers strong numerical evidences to that the kernel-interpolation estimator $f_{D,\gamma^*}=f_{D,\gamma^*,0}$ has the minimal RMSE on testing sets and shows the redundancy of the regularization term for high dimensional data interpolation.

\subsection{Real world data experiments }
In this subsection, we pursue the excellent performance for kernel interpolation on a high-dimensional real data set  MNIST, which is a    hand-writing recognition application  and is regarded as a benchmark of high dimensional data application. The basic experimental setting of this paper is the same as that in  \cite{Liang2020}.   As shown in Figure \ref{fig:sim-5}, the input of the data is a    hand-writing digits in $\{0,\dots,10\}$ and the prediction is also from  $\{0,\dots,10\}$.  Our  experiment considers the following problem: for each pair of distinct digits  $\{(i, j)\}$, $i, j \in\{0,1, \ldots, 9\}\}$,  label one digit as $1$ and the other as $0$, then fit the kernel quasi-interpolation  with Gaussian kernel $G_\gamma(x, x')=\exp\left(-\frac{\gamma\|x - x'\|^2}{d}\right)$ and $\gamma=2d=2\times784$.
\begin{figure}
	\centering
   \includegraphics[scale=0.33]{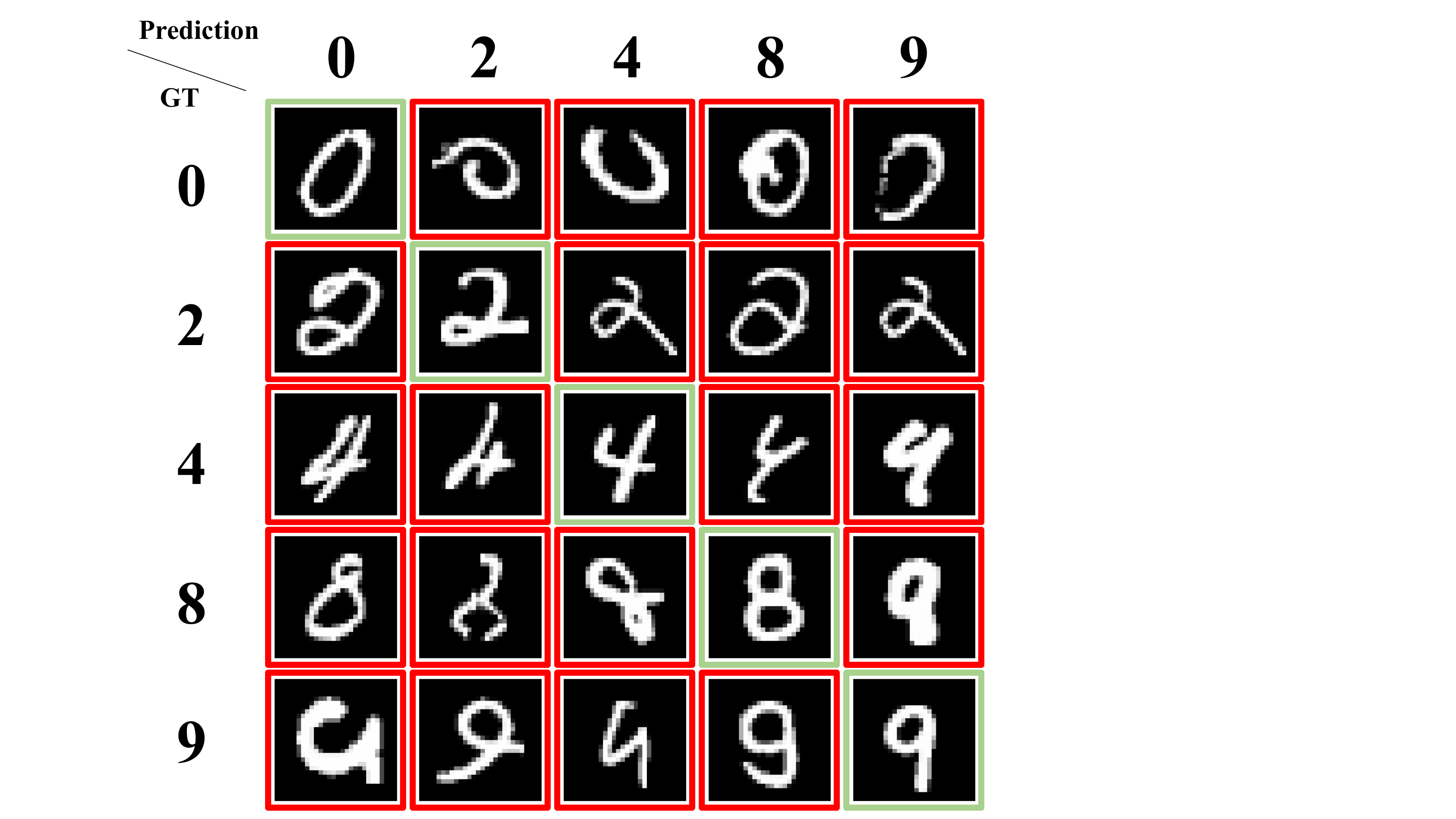}
	\caption{Task of MNIST data set}
	\label{fig:sim-5}
\end{figure}

For each of the $C(10, 2) = 45$ pairs of experiments, we chose  a finer grid of regularization $\lambda\in [0,20], $ where $\lambda $ is the regularization parameter.  We evaluate the performance on the out-of-sample test dataset, with the relative mean square error   $\frac{\sum_{i}\left(\widehat{f}\left(x_{i}\right)-y_{i}\right)^{2}}{\sum_{i}\left(\bar{y}-y_{i}\right)^{2}}$ to measure the accuracy of interpolation.
For experiment, both the training size and test size are $10,000$. The numerical results can be found in Figure \ref{fig:sim-6}.

\begin{figure}
	\centering
	\subfigure[]{\includegraphics[scale=0.33]{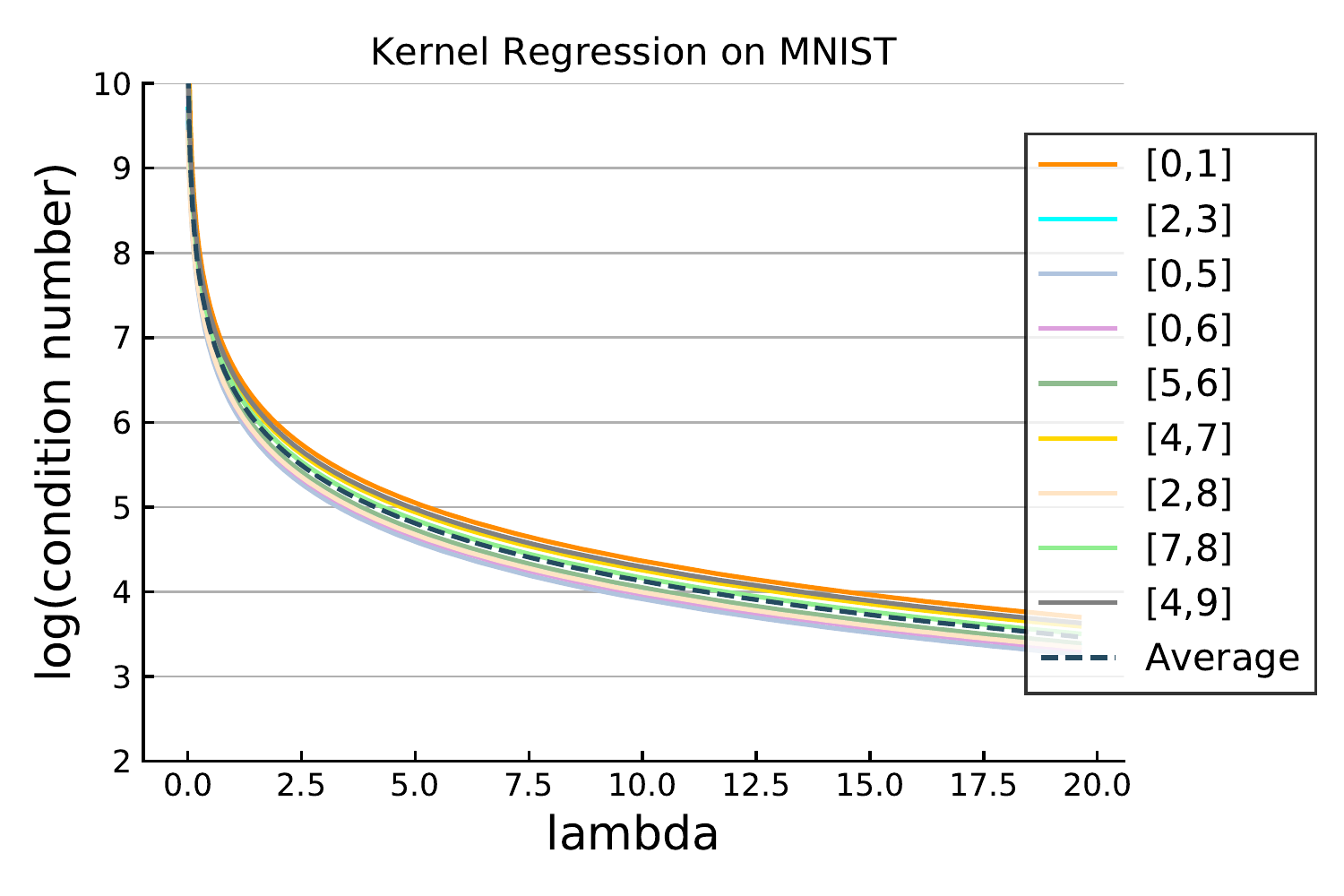}}
	\subfigure[]{\includegraphics[scale=0.33]{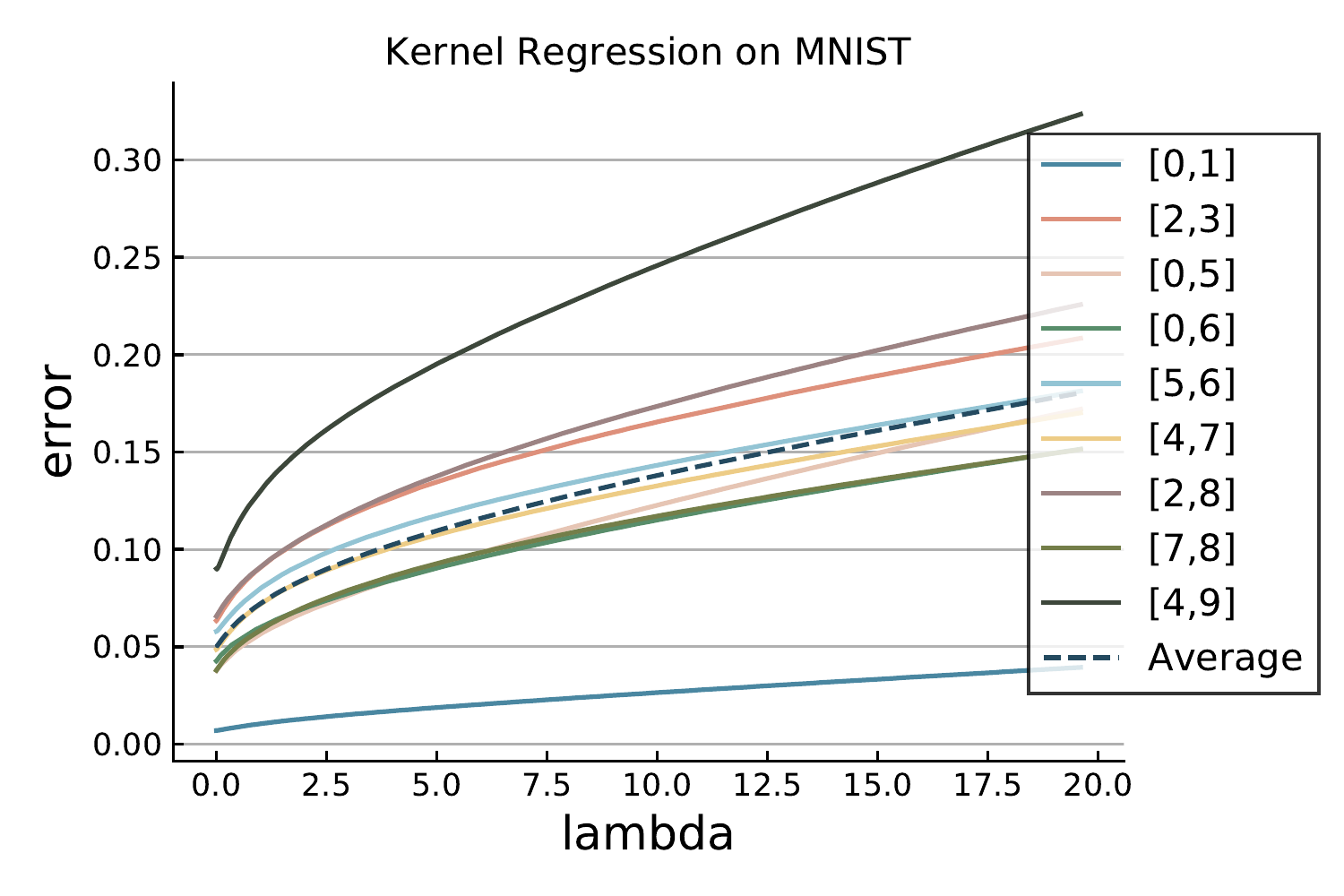}}
	\caption{Kernel interpolation on MNIST}
	\label{fig:sim-6}
\end{figure}

Figure \ref{fig:sim-6} (a) demonstrates the relation between the condition number of the quasi-interpolation matrix generated by (\ref{quasi-regularization}) and the regularization $\lambda$. It is obvious that the condition number decreases with respect to $\lambda$, showing that adding the regularization term  stabilizes the interpolation. However, as shown in  Figure \ref{fig:sim-6} (b),   among all $45$ experiments, kernel interpolation without  regularization ($\lambda = 0$) remarkably performs the best. Furthermore, the above numerical results shows that the error increases with respect to $\lambda$, which implies that the regularization term plays a negative effect in kernel interpolation, neglecting its positive effect in enhancing the stability. The main reason is that for high dimensional data interpolation, the condition number, as shown in Figure \ref{fig:sim-6} (a) is already very small. In this way, adding the regularization term plays an incremental role in guarantee the well-conditioness of the matrix. However, the additional regularization term inevitably degrades the approximation performance of kernel interpolation. All these show the effectiveness of kernel interpolation and the redundancy of the regularization term in tackling high dimensional data.

\section{Proofs}\label{Sec.proof}
In this section, we present the proofs of our main results. 

\subsection{Proofs of results stated in Section \ref{Sec.operator.difference}}
The proof of Propositions \ref{Proposition:null-space} is somewhat standard. We present the proof for the sake of completeness. 
\begin{proof}[Proof of Proposition \ref{Proposition:null-space}]
  Let
$
S_D=U\Sigma V^T
$
be the SVD of $S_D$, where
$\Sigma$ is a diagonal  matrix, $U:\mathbb R^m\rightarrow\mathbb R^m$, $V:\mathbb R^m\rightarrow\mathcal H_K$ satisfying $U^TU=V^TV=I$.
Then, we have
$$
L_{K,D}= S_D^TS_D=V\Sigma^2V^T,\qquad   \frac1m\mathbb K=S_DS_D^T=U\Sigma^2U^T,
$$
and
\begin{equation}\label{Projection}
P_m:=S_D^T(S_DS_D^T)^{-1}S_D=VV^T.
\end{equation}
 Equation  (\ref{Projection}) implies 
that $P_m$ is the projection operator from $\mathcal H_K$ to
$
\mathcal H_{K,m}=\mbox{span}\{\phi_1^{D},\dots,\phi_m^{D}\},
$
which implies
\begin{equation}\label{factor}
P_m^s=P_m  \qquad\mbox{and} \ (I-P_m)^s=(I-P_m), \qquad \ \forall \ s \in \Z_+.
\end{equation}
It follows from  (\ref{model-1}) and (\ref{operator-rep1}) that
\begin{equation}\label{operator-rep}
f_D=S_D^T(S_DS_D^T)^{-1}y_D= S_D^T(S_DS_D^T)^{-1}S_Df^*=P_mf^*.
\end{equation}
By writing
$$
L_{K,D}f_D=S_D^TS_DS_D^T(S_DS_D^T)^{-1}S_Df^*=L_{K,D}f^*,
$$
This implies (\ref{Null-space}) and completes the proof of Proposition \ref{Proposition:null-space}.
\end{proof}

To prove Theorem \ref{Theorem:noise-less-operator}, we need the following bounds related to the projection operator, whose proof is also given for the sake of completeness.

\begin{proposition}\label{Proposition:operator-bound-Pm}
	Let $\nu,\lambda\geq 0$ and $j\in\mathbb N$. We have
	\begin{equation}\label{norm-1}
	\|L_K^{v}(I-P_m)\|\leq \lambda^{1/2}\|L_K^{v}(L_{K,D}+\lambda I)^{-1/2}\|,
	\end{equation}
	and
	\begin{equation}\label{proj-operator-norm-1}
	\|S_D^T(S_DS_D^T)^{-j}S_D\|\leq  (\sigma^D_{m})^{-j+1}.
	\end{equation}
\end{proposition}

\begin{proof}
	For arbitrary $\lambda>0$, since
	\begin{eqnarray}\label{regularized}
	&&(L_{K,D}+\lambda I)^{-1}S_D^T
	= (S_D^TS_D+\lambda I)^{-1}S_D^T(S_DS_D^T+\lambda I) (S_DS_D^T+\lambda I)^{-1}\\
	&=&
	(S_D^TS_D+\lambda I)^{-1}(S_D^TS_DS_D^T+\lambda S_D^T)(S_DS_D^T+\lambda I)^{-1}
	=
	S_D^T(S_DS_D^T+\lambda I)^{-1},  \nonumber
	\end{eqnarray}
	we then have
	\begin{eqnarray}\label{regularization-1}
	(L_{K,D}+\lambda I)^{-1}L_{K,D}
	&=& S_D^T(S_DS_D^T+\lambda I)^{-1}S_D \\
	&= & P_mS_D^T(S_DS_D^T+\lambda I)^{-1}S_DP_m.\nonumber
	\end{eqnarray}
	It follows from  (\ref{factor}), (\ref{regularization-1}) and the inequality 
	$\|(L_{K,D}+\lambda I)^{-1}L_{K,D}\|\leq 1$ that
	for any $f\in\mathcal H_K$, there holds
	\begin{eqnarray*}
		&&\langle S_D^T(S_DS_D^T+\lambda I)^{-1}S_Df,f\rangle_K
		=\langle P_m(L_{K,D}+\lambda I)^{-1}L_{K,D}P_mf,f\rangle_K\\
		&=&
		\|(L_{K,D}+\lambda I)^{-1/2}L_{K,D}^{1/2}P_mf\|_K^2
		\leq
		\|(L_{K,D}+\lambda I)^{-1/2}L_{K,D}^{1/2}\|^2 \|P_mf\|_K^2
		\leq
		\|P_mf\|_K^2\\
		&=&\langle P_mf,f\rangle_K.
	\end{eqnarray*}
	This shows that $P_m-(L_{K,D}+\lambda I)^{-1}L_{K,D}$ is a positive operator.  We then obtain from
	$\lambda(L_{K,D}+\lambda I)^{-1}=I-(L_{K,D}+\lambda I)^{-1}L_{K,D}$ that
	$ \lambda(L_{K,D}+\lambda I)^{-1}-(I-P_m)$ is  positive, which in turns implies that
	$L_K^v (\lambda(L_{K,D}+\lambda I)^{-1}-(I-P_m)) L_K^v$ is  positive. Recalling (\ref{factor}), we have
	\begin{eqnarray*}
		&&
		\lambda\|(L_{K,D}+\lambda I)^{-1/2}L_K^{v}f\|_K^2
		=\langle\lambda(L_{K,D}+\lambda I)^{-1} L_K^vf, L_K^vf \rangle_K\\
		&=&
		\langle L_K^v\lambda(L_{K,D}+\lambda I)^{-1} L_K^vf, f \rangle_K
		\geq   \langle  L_K^v(I-P_m)L_K^vf,f\rangle_K
		=\langle  L_K^{v}(I-P_m)^2L_K^{v}f, f\rangle_K\\
		&=&\langle  (I-P_m)L_K^vf, (I-P_m)L_K^vf\rangle_K
		=
		\|L_K^{v}(I-P_m)f\|_K
	\end{eqnarray*}
	holds true for any $ f\in\mathcal H_K$ and any $ v\geq0$,
	which proves (\ref{norm-1}).  Since $S_D^T(S_DS_D^T)^{-j}S_D$ is self-adjoint, direct computation yields
	$$
	(S_D^T(S_DS_D^T)^{-j}S_D)^k=
	S_D^T (S_DS_D^T)^{-jk+k-1}S_D,\qquad \forall \ k\in\mathbb N.
	$$
	This implies that
	\begin{equation}\label{iteration.1}
	\|S_D^T(S_DS_D^T)^{-j}S_D\|^k=\|(S_D^T(S_DS_D^T)^{-j}S_D)^k\|
	=\|S_D^T (S_DS_D^T)^{-jk+k-1}S_D\|.
	\end{equation}
	Note that for any ${\bf c}=(c_1,\dots,c_m)^T\in\mathbb R^m$, there holds
	$$
	\|S_D^T{\bf c}\|_K^2=\frac1{m^2}\sum_{i=1}^m\sum_{j=1}^mc_ic_j\langle K_{x_i},K_{x_j}\rangle_K
	\leq \kappa^2 \left(\frac1{m}\sum_{i=1}^m |c_i|\right)^2
	\leq
	\frac{\kappa^2}{m}\sum_{i=1}^m |c_i|^2=\frac{\kappa^2}{m}\|{\bf c}\|_{\ell^2}^2.
	$$
  The compactness of $\mathcal X$ implies $\kappa<\infty$.
	This allows us to use the fact that $\|f\|_\infty\leq \kappa\|f\|_K$ for $f\in\mathcal H_K$ to deduce
	\begin{eqnarray*}
		&&\|S_D^T (S_DS_D^T)^{-jk+k-1}S_Df\|^2_K
		\leq
		\frac{\kappa^2}{m}\|(S_DS_D^T)^{-jk+k-1}S_Df\|_{\ell^2}^2\\
		&\leq&
		\frac{\kappa^2}{m}\left( \sigma^D_m\right)^{-2jk+2k-2}\|S_Df\|_{\ell^2}^2
		\leq \kappa^2\left( \sigma^D_m\right)^{-2jk+2k-2}\frac1m\sum_{i=1}^m|f(x_i)|^2\\
		&\leq&
		\kappa^2\left( \sigma^D_m\right)^{-2jk+2k-2} \|f\|_\infty^2
		\leq
		\kappa^4\left( \sigma^D_m\right)^{-2jk+2k-2} \|f\|_K^2.
	\end{eqnarray*}
	Hence,
	$$
	\|S_D^T(S_DS_D^T)^{-j}S_D\| \leq
	{\kappa^{2/k}}\left(\sigma^D_{m}\right)^{-j+1-1/k}.
	$$
	Note that the above estimate holds true all $k \in \N$. Letting $k\rightarrow\infty$, we   derive
	$$
	\|S_D^T(S_DS_D^T)^{-j}S_D\|\leq (\sigma^D_{m})^{-j+1}.
	$$
	This completes the proof of Proposition \ref{Proposition:operator-bound-Pm}.  
\end{proof}

With the help of the above proposition, we can prove Theorem \ref{Theorem:noise-less-operator} as follows.

\begin{proof}[Proof of Theorem \ref{Theorem:noise-less-operator}] Assuming (\ref{regularity-assump}) with $r\geq 1/2$, we use the facts that  $\|Af\|\leq \|A\|\|f\|_K$,   $\|AB\|=\|BA\|\leq\|A\|\|B\|$ for positive operators $A,B$ and $f\in\mathcal H_K$ to derive
	\begin{equation}\label{app.1}
	\|L_K^u(f_D-f^*)\|_\rho
	= \|L_K^{u}(I-P_m)L_K^{r-1/2}\mathcal L_K^{1/2}h^*\|_K\leq\|L_K^{u}(I-P_m)L_K^{r-1/2}\|
	\|h^*\|_\rho.
	\end{equation}
	In the rest of the proof, we need to work on two cases: $r>1$ and $\frac12\leq r\leq1$.
	
	\noindent (i) {\it Case: $\frac12\leq r\leq1$.}
	We first use (\ref{factor}) to derive that $(I-P_m)^{2r}=(I-P_m)$. We then use (\ref{norm-1}) with $v=1/2$ and the well-known Codes inequality (\ref{Codes inequality}) (in Appendix A) with $\tau=2u$ and $\tau=2r-1$ respectively  to get
	\begin{eqnarray*}
		&&
		\|L_K^{u}(I-P_m)L_K^{r-1/2}\|
		=
		\|L_K^{u}(I-P_m)^{2r-1+2u}L_K^{r-1/2}\|\\
		&\leq&\|L_K^{u}(I-P_m)^{2u}\|\|(I-P_m)^{2r-1}L_K^{r-1/2}\|
		\leq
		\|L_K^{1/2}(I-P_m)\|^{2u}\|(I-P_m)L_K^{1/2}\|^{2r-1}\\
		&=&
		\|(I-P_m)L_K^{1/2}\|^{2r+2u-1}
		\leq
		\lambda^{r+u-1/2} \|(L_{K,D}+\lambda I)^{-1/2}(L_K+\lambda I)^{1/2}\|^{2r+2u-1}.
	\end{eqnarray*}
	Plugging the above estimate into (\ref{app.1}) and noting (\ref{Def.Q}), we have
	$$
	\|L_K^u(f_D-f^*)\|_\rho\leq \lambda^{r+u-1/2}\mathcal Q_{D,\lambda}^{2r+2u-1}\|h^*\|_\rho.
	$$
	(ii){\it Case: $r>1$.}
	We first use the triangle inequality to get
	\begin{eqnarray*}
		\|L_K^{u}(I-P_m)L_K^{r-1/2}\|
		\leq
		\|L_K^{u}(I-P_m)(L_K^{r-1/2}-L_{K,D}^{r-1/2})\|
		+
		\|L_K^{u}(I-P_m)L_{K,D}^{r-1/2}\|.
	\end{eqnarray*}
	Since $r>1$, we have $r-1/2 > 1/2$. Thus
	(\ref{Null-space}) implies
	$$
	\|L_K^{u}(I-P_m)L_{K,D}^{r-1/2}\|\leq\|L_K^{u}\|\|L_{K,D}^{r-1/2}(I-P_m)\|
	=0.
	$$
	Hence, (\ref{norm-1}) with $v=1/2$ and (\ref{lip-operator}) in Appendix A yield
	\begin{eqnarray*}
		&&\|L_K^{u}(I-P_m)L_K^{r-1/2}\|
		\leq
		\|L_K^{u}(I-P_m)(L_K^{r-1/2}-L_{K,D}^{r-1/2})\| \\
		&\leq&
		\|L_K^{u}(I-P_m)\|\|L_K^{r-1/2}-L_{K,D}^{r-1/2}\|
		\leq \|L_K^{u}(I-P_m)^{2u}\|\|L_K^{r-1/2}-L_{K,D}^{r-1/2}\|_{HS}
		\\
		&\leq&
		(r-1/2)\kappa^{r-3/2}\lambda^u\|(L_{K,D}+\lambda I)^{-1/2}(L_K+\lambda I)^{1/2}\|
		\|L_K-L_{K,D}\|_{HS}.
	\end{eqnarray*}
	Inserting the above estimate into (\ref{app.1}) and noting (\ref{Def.Q}), (\ref{Def.R}), we obtain
	$$
	\|L_K^u(f_D-f^*)\|_K
	\leq (r-1/2)\kappa^{r-3/2}\|h^*\|_\rho \lambda^u\mathcal Q_{D,\lambda}\mathcal R_D.
	$$
	This completes the proof of Theorem \ref{Theorem:noise-less-operator}.
\end{proof}

The proof of  Theorem \ref{Theorem:Stability-via-operator} requires a novel integral operator approach as well as a second-order decomposition of difference of operators developed in \cite{Lin2017,Guo2017}. We proceed it as follows.

\begin{proof}[Proof of Theorem \ref{Theorem:Stability-via-operator}] Define $f_D^*=P_mf^*$. We then write,
	\begin{equation}\label{error-dec}
	\|L_K^{u}(f_D-f^*)\|_K\leq \|L_K^{u}(f_D-f_D^*)\|_K+\|L_K^{u}(f_D^*-f^*)\|_K.
	\end{equation}
	The term $\|L_K^{u}(f_D^*-f^*)\|_K$ has been dealt with in the proof of Theorem \ref{Theorem:noise-less-operator}, which allows us to concentrate on  bounding $\|L_K^{u}(f_D-f_D^*)\|_K$.
	The crux of our proof is to introduce the following second-order decomposition for  differences of operators. A prototype of this decomposition can be found in  \cite{Lin2017,Guo2017}.
	Let $A,B$ be invertible  operators. We first write
	\begin{equation}\label{first-order-dec}
	A^{-1}-B^{-1}=A^{-1}(B-A)B^{-1}=B^{-1}(B-A)A^{-1}.
	\end{equation}
	We then use (\ref{first-order-dec})  to write
	\begin{equation}\label{Second-order-dec}
	A^{-1}-B^{-1} = B^{-1}(B-A)B^{-1}+
	B^{-1}(B-A)A^{-1}(B-A)B^{-1}.
	\end{equation}
	Setting $A= S_D S_D^T $  and $B=S_DS_D^T+\mu I$ with $\mu>0$ in (\ref{Second-order-dec}), we obtain that
	\begin{eqnarray*}
		(S_D S_D^T)^{-1}
		&=& (S_DS_D^T+\mu I)^{-1}+ \mu(S_DS_D^T+\mu I)^{-2}\\
		&+&
		\mu^2(S_DS_D^T+\mu I)^{-1} (S_DS_D^T)^{-1}(S_DS_D^T+\mu I)^{-1}.
	\end{eqnarray*}
	Note that  (\ref{regularized}) implies
	\begin{equation}\label{exchange}
	S^T_D(S_DS_D^T+\mu I)^{-k}=(L_{K,D}+\mu I)^{-k}S_D^T,\qquad\forall\  k\in\mathbb N.
	\end{equation}
	Hence, we have
	\begin{eqnarray*}
		&&S_D^T(S_DS_D^T)^{-1} (y_D-S_Df^*)
		= S_D^T(S_DS_D^T+\mu I)^{-1}(y_D-S_Df^*)\\
		&+&
		\mu  S_D^T(S_DS_D^T+\mu I)^{-2} (y_D-S_Df^*)\\
		& +&
		\mu^2S_D^T(S_DS_D^T+\mu I)^{-1} (S_DS_D^T)^{-1}(S_DS_D^T+\mu I)^{-1}(y_D-S_Df^*) \\
		&=&
		(L_{K,D}+\mu I)^{-1}(S_D^Ty_D-L_{K,D}f^*)+
		\mu(L_{K,D}+\mu I)^{-2}(S_D^Ty_D-L_{K,D}f^*)\\
		&+&
		\mu^2 (L_{K,D}+\mu I)^{-1} S_D^T(S_DS_D^T)^{-2}S_D(L_{K,D}+\mu I)^{-1}(S_D^Ty_D-L_{K,D}f^*).
	\end{eqnarray*}
	Therefore, it follows from (\ref{Def.P}) and (\ref{Def.Q})  that
	\begin{eqnarray*}
		&&\|L_K^{u}(f_D- f^*_D)\|_K=\|L_{K}^{u}S_D^T(S_DS_D^T)^{-1}(y_D-S_Df^*)\|_K \nonumber\\
		&\leq&
		\|L_K^{u}(L_{K,D}+\mu I)^{-1}(S_D^Ty_D-L_{K,D}f^*)\|_K
		+\mu\|L_K^{u}(L_{K,D}+\mu I)^{-2}(S_D^Ty_D-L_{K,D}f^*)\|_K \nonumber\\
		&+&
		\mu^2\|L_K^{u}(L_{K,D}+\mu I)^{-1} S_D^T(S_DS_D^T)^{-2}S_D(L_{K,D}+\mu I)^{-1}(S_D^Ty_D-L_{K,D}f^*)\|_K \nonumber\\
		&\leq&
		2\mu^{u-1/2}\mathcal Q^2_{D,\mu} \mathcal P_{D,\mu}+\mu^{u+1/2} \mathcal Q^{2u+1}_{D,\mu}\mathcal P_{D,\mu}\|S_D^T(S_DS_D^T)^{-2}S_D\|.
	\end{eqnarray*}
	Applying inequality (\ref{proj-operator-norm-1}) with $j=2$, we have
$$
	\|L_K^u(f_D- f^*_D)\|_K
	\leq
	(2\mu^{u-1/2}+ {\mu^{u+1/2}} (\sigma^D_{m})^{-1}\mathcal Q^{2u-1}_{D,\mu}) \mathcal Q^{2}_{D,\mu} \mathcal P_{D,\mu},\qquad\forall\ \mu>0.
$$
	Inserting the above estimate into (\ref{error-dec}) and noting (\ref{noise-less-operator}), we derive (\ref{noise-operator}) directly. This completes the proof of Theorem \ref{Theorem:Stability-via-operator}.  
\end{proof}

The main difficulty in the proof of Theorem                             \ref{Theorem:native-barrier-without-noise} is to find a good
projection of $f^*$ onto $\mathcal H_K$ and quantify the error.
In our proof, we construct  a $f_\lambda\in\mathcal H_K$  and treat
$f^*-f_\lambda$ as the noise. In such a way, the basic idea is   similar as that in the proof of Theorem \ref{Theorem:Stability-via-operator} by involving $\sigma_m^D$ to reflect the stability of kernel interpolation.

\begin{proof}[Proof of Theorem \ref{Theorem:native-barrier-without-noise}]
	For an arbitrary $\lambda>0$, define
	$$
	f_\lambda=( L_K+\lambda I)^{-1}\mathcal L_Kf^*,  \qquad
	f_{D,\lambda}= ( L_{K,D}+\lambda I)^{-1}\mathcal L_{K,D}f^*.
	$$
	Assuming that (\ref{regularity-assump}) holds true with $0\leq r\leq 1$,  Smale and Zhou \cite{Smale2005} (or \cite{Smale2007}) proved that
	\begin{equation}\label{approximation-error-small1}
	\|f^*-f_\lambda\|_\rho\leq \lambda^r\|h^*\|_\rho.
	\end{equation}
	Thus, we have
	\begin{equation}\label{error-decomposition-out}
	\|f^*-f_D\|_\rho\leq \lambda^r\|h^*\|_\rho
	+\|f_\lambda-f_{D,\lambda}\|_\rho+\|f_{D,\lambda}-f_D\|_\rho.
	\end{equation}
	By writing
	\begin{eqnarray*}
		&&f_\lambda-f_{D,\lambda}
		=(L_K+\lambda I)^{-1}\mathcal L_Kf^*-(L_{K,D}+\lambda I)^{-1}\mathcal L_{K,D}f^* \nonumber\\
		&=&[(L_K+\lambda I)^{-1}-(L_{K,D}+\lambda I)^{-1}]\mathcal L_Kf^*
		+(L_{K,D}+\lambda I)^{-1}(\mathcal L_{K,D}f^*-\mathcal L_Kf^*) \nonumber\\
		&=&
		(L_{K,D}+\lambda I)^{-1}(L_{K,D}-L_K)(L_K+\lambda I)^{-1}\mathcal L_Kf^*
		+(L_{K,D}+\lambda I)^{-1}(\mathcal L_{K,D}f^*-\mathcal L_Kf^*).
	\end{eqnarray*}
	we obtain
	\begin{eqnarray}\label{approximation-error-small2}
	 \|f_\lambda-f_{D,\lambda}\|_\rho&\leq&
	\|L_K^{1/2} (L_{K,D}+\lambda I)^{-1}(L_{K,D}-L_K)(L_K+\lambda I)^{-1}L_K^{1/2+r}\mathcal L_K^{1/2}h^*\|_K \nonumber \\
	&+&
	\|L_K^{1/2}(L_{K,D}+\lambda I)^{-1}(\mathcal L_{K,D}f^*-\mathcal L_Kf^*)\|_K\nonumber\\
	&\leq&
	\lambda^{r-1/2} \mathcal Q_{D,\lambda}^2\mathcal W_{D,\lambda}\|h^*\|_\rho
	+ \mathcal Q^2_{D,\lambda}\mathcal U_{D,\lambda,f^*}.
	\end{eqnarray}
	It follows from (\ref{Second-order-dec}) with  $A= S_D S_D^T $  and $B=S_DS_D^T+\lambda  I$  and (\ref{exchange}) that
	\begin{eqnarray*}
		&& f_D-f_{D,\lambda}
		=S_D^T((S_DS_D^T)^{-1}-(S_DS_D^T+\lambda)^{-1})\mathcal S_Df^*\\
		&=&
		\lambda  S_D^T(S_DS_D^T+\lambda I)^{-2}\mathcal S_Df^*
		+
		\lambda^2S_D^T(S_DS_D^T+\lambda I)^{-1} (S_DS_D^T)^{-1}(S_DS_D^T+\lambda I)^{-1}\mathcal S_Df^* \\
		&=&
		\lambda(L_{K,D}+\lambda I)^{-2}\mathcal L_{K,D}f^*
		+
		\lambda^2 (L_{K,D}+\lambda I)^{-1} S_D^T(S_DS_D^T)^{-2}S_D(L_{K,D}+\lambda I)^{-1}\mathcal L_{K,D}f^*\\
		&=&
		\lambda(L_{K,D}+\lambda I)^{-2}(\mathcal L_{K,D}-\mathcal L_K)f^*+\lambda(L_{K,D}+\lambda I)^{-2}\mathcal L_Kf^*\\
		&+&
		\lambda^2 (L_{K,D}+\lambda I)^{-1} S_D^T(S_DS_D^T)^{-2}S_D(L_{K,D}+\lambda I)^{-1}(\mathcal L_{K,D}-\mathcal L_K)f^*\\
		&+&
		\lambda^2 (L_{K,D}+\lambda I)^{-1} S_D^T(S_DS_D^T)^{-2}S_D(L_{K,D}+\lambda I)^{-1}\mathcal L_Kf^*.
	\end{eqnarray*}
	Using  (\ref{Def.Q}), (\ref{Def.U}), and (\ref{Codes inequality}) (Appendix A),  and the fact that $\|Af\|_K\leq \|A\|\|f\|_K$ for positive operator $A$ and $f\in\mathcal H_K$, we derive
	\begin{eqnarray*}
		&&\|\lambda(L_{K,D}+\lambda I)^{-2}(\mathcal L_{K,D}-\mathcal L_K)f^*\|_\rho
		=
		\lambda\|L_K^{1/2}(L_{K,D}+\lambda I)^{-2}(\mathcal L_{K,D}-\mathcal L_K)f^*\|_K\\
		&\leq &
		\lambda \mathcal Q_{D,\lambda}\lambda^{-1}\mathcal Q_{D,\lambda}
		\mathcal U_{D,\lambda,f^*}
		=\mathcal Q_{D,\lambda}^2\mathcal U_{D,\lambda,f^*},\\
		&&\|\lambda(L_{K,D}+\lambda I)^{-2}\mathcal L_Kf^*\|_\rho
		\leq\lambda\|L_K^{1/2}(L_{K,D}+\lambda I)^{-2}  L_K^{1/2+r}\|\|\mathcal L_K^{1/2}h^*\|_K\\
		&\leq&
		\lambda\mathcal Q_{D,\lambda} \lambda^{-2+1+r}  \mathcal Q_{D,\lambda}^{2r+1}\|h^*\|_\rho
		=\|h^*\|_\rho\lambda^r\mathcal Q_{D,\lambda}^{2r+2}.
	\end{eqnarray*}
	Incorporating
	(\ref{proj-operator-norm-1}), we further derive that
	\begin{eqnarray*}
		&&
		\|\lambda^2 (L_{K,D}+\lambda I)^{-1} S_D^T(S_DS_D^T)^{-2}S_D(L_{K,D}+\lambda I)^{-1}(\mathcal L_{K,D}-\mathcal L_K)f^*\|_\rho\\
		&\leq&
		\lambda^2\|L_K^{1/2}(L_{K,D}+\lambda I)^{-1}\| \|S_D^T(S_DS_D^T)^{-2}S_D\|\|(L_{K,D}+\lambda I)^{-1}(\mathcal L_{K,D}-\mathcal L_K)f^*\|_K\\
		&\leq&
		\lambda^2\mathcal Q_{D,\lambda}\lambda^{-1/2}(\sigma_m^D)^{-1}\lambda^{-1/2}\mathcal Q_{D,\lambda}\mathcal U_{D,\lambda,f^*}
		=
		\lambda(\sigma_m^D)^{-1}\mathcal Q^2_{D,\lambda}\mathcal U_{D,\lambda,f^*},\\
		&&\|\lambda^2 (L_{K,D}+\lambda I)^{-1} S_D^T(S_DS_D^T)^{-2}S_D(L_{K,D}+\lambda I)^{-1}\mathcal L_Kf^*\|_\rho\\
		&\leq&
		\lambda^2 \|L_K^{1/2}(L_{K,D}+\lambda I)^{-1} \| \|S_D^T(S_DS_D^T)^{-2}S_D\|(L_{K,D}+\lambda I)^{-1} L_K^{1/2+r}\|\|\mathcal L^{1/2}h^*\|_K\\
		&\leq&
		\lambda^2\mathcal Q_{D,\lambda}\lambda^{-1/2}(\sigma_m^D)^{-1}\lambda^{-1/2+r}\mathcal Q_{D,\lambda}^{2r+1}\|h^*\|_\rho
		=
		\lambda^{1+r}(\sigma_m^D)^{-1}\mathcal Q_{D,\lambda}^{2r+2}\|h^*\|_\rho.
	\end{eqnarray*}
	It therefore follows that
	\begin{equation}\label{approximation-error-small3}
	\| f_D-f_{D,\lambda}\|_\rho
	\leq
	(1+\lambda(\sigma_m^D)^{-1}) (\mathcal Q_{D,\lambda}^2\mathcal U_{D,\lambda,f^*}+\|h^*\|_\rho\lambda^r\mathcal Q_{D,\lambda}^{2r+2}).
	\end{equation}
	Plugging (\ref{approximation-error-small1}) and (\ref{approximation-error-small2}) into (\ref{approximation-error-small3}),
	we have
	\begin{eqnarray*}
		&&\|f^*-f_D\|_\rho\leq \lambda^r\|h^*\|_\rho+\lambda^{r-1/2} \mathcal Q_{D,\lambda}^2\mathcal W_{D,\lambda}\|h^*\|_\rho
		+ \mathcal Q_{D,\lambda}\mathcal U_{D,\lambda,f^*}\\
		&+&
		(1+\lambda(\sigma_m^D)^{-1}) (\mathcal Q_{D,\lambda}^2\mathcal U^2_{D,\lambda,f^*}+\|h^*\|_\rho\lambda^r\mathcal Q_{D,\lambda}^{2r+2}),
	\end{eqnarray*}
	which is the desired result.  
\end{proof}

\subsection{Proofs of results stated in Section \ref{Sec.Random}}
In this part, we aim to prove results in Section \ref{sec.spectrum}. The main tools in our analysis are tight bounds for operator differences presented in Appendix B.  Combine Theorem \ref{Theorem:noise-less-operator} with Appendix B, we can prove Theorem \ref{Theorem:native-barrier-without-noise-eigenvalue} as follows.

\begin{proof}[Proof of Theorem \ref{Theorem:noise-less-eigenvalue}]   If $\frac12\leq r\leq 1$, we get from
	Lemma \ref{Lemma:Q} in Appendix B  that with confidence $1-\delta$, there holds
	\begin{eqnarray*}
		\mathcal Q_{D,\lambda}^{2r}
		\leq    2^r\left( 2\kappa(\kappa+8) \mathcal A_{D,\lambda}+1 \right)^{2r} \log^{4r}\frac8\delta.
	\end{eqnarray*}
	Then, it follows from Theorem \ref{Theorem:noise-less-operator} with $u=1/2$ that with confidence $1-\delta$, there holds
	$$
	\|f_D-f^*\|_\rho\leq \min_{\lambda>0}2^{r}\|h^*\|_\rho\lambda^r
	\left( 2\kappa(\kappa+8)\mathcal A_{D,\lambda}+1 \right)^{2r} \log^{4r}\frac8\delta.
	$$
	If $r>1$, then Lemma \ref{Lemma:operator-difference} and Lemma \ref{Lemma:Q} in Appendix B  show  that with confidence $1-\delta$, there holds
	$$
	\mathcal Q_{D,\lambda}\mathcal R_{D}
	\leq
	\frac{2\sqrt{2}\kappa^2}{\sqrt{m}}   ( 2\kappa(\kappa+8) \mathcal A_{D,\lambda}+1 )  \log^3\frac{16}\delta.
	$$
	Plugging the above estimate into (\ref{noise-less-operator}) with $u=1/2$, we have
	$$
	\|f_D-f^*\|_\rho
	\leq
	2\sqrt{2}(r-1/2)\kappa^{r+1/2}\|h^*\|_\rho\log^3\frac{16}\delta  \min_{\lambda>0}  \sqrt{\frac{\lambda}{m}}( 2\kappa(\kappa+8) \mathcal A_{D,\lambda}+1).
	$$
	This completes the proof of Theorem \ref{Theorem:noise-less-eigenvalue} with $$
	C_1:=\left\{\begin{array}{ll}
	2^{r}\|h^*\|_\rho(\max\{2\kappa(\kappa+8),1\})^{2r},& \quad \frac12\leq r\leq 3/2,\\
	2\sqrt{2}(r-1/2)\kappa^{r+1/2}\|h^*\|_\rho\max\{2\kappa(\kappa+8),1\},&
	\quad  r>3/2.
	\end{array}
	\right.
	$$ 
\end{proof}

Similar as above, we can prove Theorem \ref{Theorem:Stability-via-eigenvalue} via Theorem \ref{Theorem:noise-less-eigenvalue} and Appendix B.

\begin{proof}[Proof of Theorem \ref{Theorem:Stability-via-eigenvalue}]
	Due to (\ref{model-2}), we obtain $|y_i|\leq \|f^*\|_\infty+\gamma$. Then,
	it follows from Lemma \ref{Lemma:Q} in Appendix with $M=\|f^*\|_\infty+\gamma$ that with confidence $1-\delta$, there holds
	$$
	\mathcal Q_{D,\mu}^2\mathcal P_{D,\mu}
	\leq
	8(\|f^*\|_\infty+\gamma)( 2\kappa(\kappa+8)\mathcal A_{D,\mu}+1)^2
	(\kappa+8)\sqrt{\mu}\mathcal A_{D,\mu} \log^6\frac{16}{\delta}.
	$$
	This together with Theorem \ref{Theorem:Stability-via-operator} for $u=1/2$, (\ref{operator-matrix-eigenvalue}) and  (\ref{noise-less-eigenvalue}) completes the proof of Theorem \ref{Theorem:Stability-via-eigenvalue} with 
	$$
	C_2:=16(\|f^*\|_\infty+\gamma)(\kappa+8)\max\{2\kappa(\kappa+8),1\}
	$$.  
\end{proof}

The proof Theorem \ref{Theorem:native-barrier-without-noise-eigenvalue} is trivial and can be derived from Theorem \ref{Theorem:native-barrier-without-noise} and Appendix directly.

\begin{proof}[Proof of Theorem \ref{Theorem:native-barrier-without-noise-eigenvalue}]
	It follows from Lemma \ref{Lemma:Q} that with confidence $1-\delta$, there holds
	\begin{eqnarray*}
		\mathcal Q_{D,\lambda}^{2r+2}&\leq& 2^{r+1}
		(2\kappa(\kappa+8)\mathcal A_{D,\lambda}+1)^{2r+2}    \log^{4r+4}\frac{24}\delta,\\
		\mathcal Q^2_{D,\lambda}\mathcal U_{D,\lambda,f^*}
		&\leq&
		\frac{4\|f\|_\infty}{\kappa}(2\kappa(\kappa+8)\mathcal A_{D,\lambda}+1)^{2} \sqrt{\lambda}(\kappa+8)\mathcal A_{D,\lambda}\log^6\frac{24}{\delta},\\
		\mathcal Q_{D,\lambda}^2\mathcal W_{D,\lambda}&\leq&
		4\kappa
		(2\kappa(\kappa+8)\mathcal A_{D,\lambda}+1)^{2}\sqrt{\lambda}(\kappa+8)\mathcal A_{D,\lambda}  \log^{6}\frac{24}\delta.
	\end{eqnarray*}
	Plugging the above three estimates into  (\ref{Native-barrier-error}) with $u=1/2$, we obtain (\ref{Native-barrier-error-eigenvalue}) to complete the proof of Theorem \ref{Theorem:Stability-via-eigenvalue} with 	
	$$
	C_{3}:=(\max\{2\kappa(\kappa+8),1\})^{2r+2}\}\max\{2^{r+1}\|h^*\|_\rho,
	(2\|f\|_\infty/\kappa+\kappa\|h^*\|_\rho)2\kappa (\kappa+8)\}.
	$$ 
\end{proof}

\subsection{Proofs of results stated in Section \ref{sec.spectrum}} 
In this part, we focus on proving results concerning eigen-values. The main idea in the proof of Proposition \ref{Proposition:uniform-spectrum} is a close relation between effective dimension and its empirical counterpart. 
\begin{proof}[Proof of Proposition \ref{Proposition:uniform-spectrum}]
	For an arbitrary $0 \le \lambda \le 1$, define the effective dimension  and  empirical effective dimension \cite{Mucke2018} to be
	\begin{equation}\label{Effective-dimension}
	\mathcal{N}(\lambda)={\rm Tr}((\lambda I+L_K)^{-1}L_K),  \qquad
	\mathcal{N}_D(\lambda)={\rm Tr}((\lambda I+L_{K,D})^{-1}L_{K,D}).
	\end{equation}
	Since
	$$
	\mathcal N_{D}(\lambda)={\rm Tr}[(L_{K,D}+\lambda I)^{-1}L_{K,D}]
	={\rm Tr}[(\lambda m I+\mathbb K)^{-1}\mathbb K],
	$$
	we have by (\ref{Def.A}) that
	$$
	\mathcal A_{D,\lambda}=
	\left(\frac{1}{m\lambda}+\frac{1}{\sqrt{m\lambda}}\right)\max\left\{1,\sqrt{ \mathcal N_D(\lambda)}\right\}.
	$$
	Furthermore, Lemma \ref{Lemma:Effective dimension} (in Appendix B) asserts that with confidence $1-\delta$, there holds
	$$
	\sqrt{\max\{\mathcal N_D(\lambda),1\}}
	\leq   17\left(1+\frac{1}{m\lambda}\right)\sqrt{\max\{\mathcal
		N(\lambda),1\}}\log^2\frac{4}{\delta},
	$$
	implying
	\begin{equation}\label{bound-A-1}
	\mathcal A_{D,\lambda}
	\leq
	17\left(\frac{1}{m\lambda}+\frac{1}{\sqrt{m\lambda}}\right)\left(1+\frac{1}{m\lambda}\right)\sqrt{\max\{\mathcal
		N(\lambda),1\}}\log^2\frac{4}{\delta}.
	\end{equation}
	What remains in the proof is to bound the effective dimension $\mathcal N(\lambda)$ under the assumption in (\ref{Eigenvalue-decay-ass-alg}) or  (\ref{Eigenvalue-decay-ass-exp}), which we will treat separately.
	If  (\ref{Eigenvalue-decay-ass-alg}) holds true, then
	\begin{eqnarray*}
		\mathcal N(\lambda)
		&=&
		\sum_{\ell=1}^\infty\frac{\sigma_\ell}{\lambda+\sigma_\ell}
		\leq
		\sum_{\ell=1}^\infty\frac{c_0\ell^{-\beta}}{\lambda+c_0\ell^{-\beta}}
		=
		\sum_{\ell=1}^\infty\frac{c_0}{c_0+\lambda\ell^{\beta}}\\
		&\leq&
		\int_0^\infty\frac{c_0}{c_0+\lambda t^{\beta} }dt
		\leq
		c_1 \lambda^{-1/\beta},
	\end{eqnarray*}
	where $c_1$ is a constant depends only on $c_0$. Plugging  the above estimate into (\ref{bound-A-1}), we have, with confidence $1-\delta$,
	$$
	\mathcal A_{D,\lambda}
	\leq
	17\sqrt{c_1}\left(\frac{1}{m\lambda}+\frac{1}{\sqrt{m\lambda}}\right)\left(1+\frac{1}{m\lambda}\right)\lambda^{-1/(2\beta)}\log^2\frac{4}{\delta}.
	$$
	If  (\ref{Eigenvalue-decay-ass-exp}) holds true, then
	\begin{eqnarray*}
		\mathcal N(\lambda)
		\leq
		\sum_{\ell=1}^\infty\frac{c_0 e^{-\alpha  \ell^{1/d}}}{\lambda+c_0 e^{-\alpha \ell^{1/d}}}
		=
		\sum_{\ell=1}^\infty\frac{c_0}{c_0+\lambda e^{\alpha\ \ell^{1/d}}}
		\leq
		\int_0^\infty\frac{c_0}{c_0+\lambda  e^{ \alpha\ t^{1/d}}}dt
		\leq c_2 d! \ \alpha^{-d}\  \log^d \frac1\lambda, 
	\end{eqnarray*}
	where $c_2$ is a constant depending only on $c_0$. In deriving the last inequality above, we have used Lemma \ref{multi-log-esti} (in Appendix B). Substituting the last inequality above into (\ref{bound-A-1}), we have, with confidence $1-\delta$,
	$$
	\mathcal A_{D,\lambda}
	\leq
	17c_2 \sqrt{d!}\ \alpha^{-d/2}\left(\frac{1}{m\lambda}+\frac{1}{\sqrt{m\lambda}}\right)\left(1+\frac{1}{m\lambda}\right)\log^{d/2}\frac1\lambda \log^2\frac{4}{\delta}.
	$$
	This completes the proof of Proposition \ref{Proposition:uniform-spectrum}.
\end{proof}

The proofs of Corollary \ref{Corollary:noise-less-rate} and Corollary \ref{Corollary:noise-less-rate-G}  can be directly derived from Theorem \ref{Theorem:noise-less-eigenvalue} and Proposition \ref{Proposition:uniform-spectrum}.

\begin{proof}[Proof of Corollary \ref{Corollary:noise-less-rate}]
	It follows from Proposition \ref{Proposition:uniform-spectrum} with $\beta=2\tau/d$ that
	\begin{equation}\label{A-for-sobolev-spline}
	\mathcal A_{D,\lambda}\leq  C_1 \left(\frac{1}{m\lambda}+\frac{1}{\sqrt{m\lambda}}\right)\left(1+\frac{1}{m\lambda}\right)
	\lambda^{-d/(4\tau)}\log^2\frac{4}{\delta}
	\end{equation}
	holds with confidence $1-\delta$.
	Let $\lambda=m^{-\frac{2\tau}{2\tau+d}}$, we have with confidence $1-\delta$ that
	$$
	\mathcal A_{D,\lambda}\leq4C_1\log^2\frac{4}{\delta}.
	$$
	The desired result follows from the above inequality and  (\ref{noise-less-eigenvalue}).
	This completes the proof of Corollary \ref{Corollary:noise-less-rate}.
\end{proof}

\begin{proof}[Proof of Corollary \ref{Corollary:noise-less-rate-G}]
	It follows from Proposition \ref{Proposition:uniform-spectrum} with $\alpha=a$ that
	\begin{equation}\label{A-for-sobolev-spline1}
	\mathcal A_{D,\lambda}\leq  C_4\sqrt{d}a^{-d/2} \left(\frac{1}{m\lambda}+\frac{1}{\sqrt{m\lambda}}\right)\left(1+\frac{1}{m\lambda}\right)\log^{d/2}\frac1\lambda
	\log^2\frac{4}{\delta}
	\end{equation}
	holds with confidence $1-\delta$.
	Let $\lambda=m^{-1}$, we have with confidence $1-\delta$ that
	$$
	\mathcal A_{D,\lambda}\leq4C_4\sqrt{d}a^{-d/2}\log^{d/2}m\log^2\frac{4}{\delta}.
	$$
	The desired result follows from the above inequality and  (\ref{noise-less-eigenvalue}).
	This completes the proof of Corollary \ref{Corollary:noise-less-rate-G}.  
\end{proof} 
\section*{Appendix A: Positive Operator Theory}
We  include here several definitions and properties of positive operators. We refer readers to \cite{Bhatia2013} for more details.
For two Hilbert spaces $\mathcal H_1,\mathcal H_2$, denote by
$\mathcal L(\mathcal H_1,\mathcal H_2)$   the space of all
bounded linear operators from $\mathcal H_1$ to $\mathcal H_2$.
Given a  linear operator $A$, its adjoint operator, denoted by $A^T\in \mathcal
L(\mathcal H_2,\mathcal H_1)$ is defined to be

$$
\langle A^Tg,f\rangle_{\mathcal H_1}=  \langle g, Af\rangle_{\mathcal H_2}
$$
for any $f\in\mathcal H_1$ and $g\in\mathcal H_2$. Denote
$\mathcal L(\mathcal H)=\mathcal L(\mathcal H,\mathcal H)$.
For $A\in\mathcal L(\mathcal H)$, define its operator norm as
$$
\|A\|=\sup_{\|f\|_{\mathcal H}=1}\|Af\|_{\mathcal H}.
$$

We say an operator $A\in\mathcal L(\mathcal H)$ to be  self-adjoint, if $A=A^T$. Furthermore, we say an operator $A$ to be positive, if $A\in\mathcal L(\mathcal H)$ is self-adjoint and $\langle f,Af\rangle_{\mathcal H}\geq 0$ for all $f\in\mathcal H$.
A bounded linear  operator $A$  is said to be compact, if  the image under $A$ of any bounded subset of $\mathcal H$    is a relatively compact subset  (has compact closure) of $\mathcal H$.
If $A$ is compact and positive, then
there exists a normalized eigenpairs
of $A$, denoted by $\{(\sigma_\ell,\phi_\ell)\}_{\ell=1}^\infty$ with
$\sigma_1\geq\sigma_2\geq\dots\geq0$ and
$\{\phi_\ell\}_{\ell=1}^\infty$ forming an orthonormal basis of
$\mathcal H$.
For $F:\mathbb R_+\cap \{0\}\rightarrow \mathbb R,$   define
$$
F(A)=\sum_{\ell=1}^\infty F(\sigma_\ell)\phi_\ell\otimes\phi_\ell
=\sum_{\ell=1}^\infty F(\sigma_\ell)\langle\cdot,\phi_\ell\rangle_{\mathcal H}\phi_\ell.
$$
The trace and  Hilbert-Schmidt norm  of the positive operator  $A$
is denoted by
$$
\mbox{Tr}(A)=\sum_{i=1}^\infty \sigma_i,
$$
and
$$
\|A\|_{HS}=(\mbox{Tr}(A^2))^{1/2}=\left(\sum_{j=1}^\infty \sigma_j^2\right)^{1/2},
$$
respectively. If $\|A\|_{HS}<+\infty$, we then call $A$ a Hilbert-Schmidt operator. If $A$ and $B$ are Hilbert-Schmidt, then
$$
\|A\| \leq \|A\|_{HS},\qquad \|AB\|_{HS}\leq \|A\|\|B\|_{HS}.
$$

In the following, we present some important properties of positive operators $A,B$, which are well known and can be found in \cite{Bhatia2013}.
Since   positive operators are always self-adjoint, we have
$$
\|AB\|=\|BA\|.
$$
If $A-B$ is also a positive operator, then
$$
\|Bf\|_{\mathcal H}\leq \|Af\|_{\mathcal H}.
$$
For any $0\leq\tau\leq 1$, the Cordes inequality \cite{Bhatia2013} shows
\begin{equation}\label{Codes inequality}
\|A^\tau B^\tau\|\leq \|AB\|^\tau.
\end{equation}
If in addition $A$ and $B$ are Hilbert-Schmidt and satisfy $\max\{\|A\|,\|B\|\}\leq \kappa$, then the Lipschitz property  yields
\begin{equation}\label{lip-operator}
\|A^\nu-B^\nu\|_{HS}\leq  \left\{\begin{array}{ll}
\nu\kappa^{\nu-1}\|A-B\|_{HS},& \quad \nu\geq 1,\\
\|A-B\|_{HS}^\nu, & \quad 0<\nu<1.
\end{array}
\right.
\end{equation}

\section*{Appendix B: Bounds for Operator Differences with Random Sampling}

Bounding $\mathcal R_D$, which can be loosely considered as the difference between $L_K$ and its empirical version $L_{K,D}$,  is a classical topic in statistical learning theory \cite{Smale2005,Smale2007,Yao2007,Caponnetto2007,Blanchard2016,Lin2017,Guo2017}.
Using the classical Bernstein inequality for Banach-valued functions in \cite{Pinelis1994}, authors of  \cite[Prop. 5.3]{Yao2007} and \cite[Lemma 17]{Lin2017} give tight upper bounds of $\mathcal R_D$ and $\mathcal W_{D,\lambda}$, which we quote in the following lemma.

\begin{lemma}\label{Lemma:operator-difference}
	Let
	$\delta\in(0,1)$ and $\lambda>0$. With confidence at least $1-\delta,$ there holds
	\begin{eqnarray}
	\mathcal R_{D}
	&\leq&
	\frac{2\kappa^2}{\sqrt{m}}\log^{1/2}\frac2\delta,\label{bound.RD}\\
	\mathcal W_{D,\lambda}
	&\leq&
	{\mathcal B}_{m, \lambda} \log \frac2\delta,
	\end{eqnarray}
	where $\mathcal N(\lambda)$ is defined by (\ref{Effective-dimension}) and
	\begin{equation}\label{Definition B}
	\mathcal B_{m,\lambda}=\frac{2\kappa}{\sqrt{m}}\left\{\frac{\kappa}{\sqrt{m\lambda}}
	+\sqrt{\mathcal{N}(\lambda)}\right\}.
	\end{equation}
\end{lemma}

The following lemma provides an upper bound for $\mathcal Q_{D,\lambda}$. A first similar result is  given in \cite{Blanchard2016} under  a mild restriction on $m$ and $\delta$. In \cite{Guo2017}, the restriction is removed by using the second order decomposition technique (\ref{Second-order-dec}). The following lemma is derived from \cite[Prop.1]{Guo2017} and Cordes inequality (\ref{Codes inequality}).

\begin{lemma}\label{Lemma:Operator-product} For any $0<\delta<1$ and $\lambda>0$, with confidence
	$1-\delta,$ there holds
	$$
	\mathcal Q_{D,\lambda}\leq
	\sqrt{2}  \left(\frac{\mathcal B_{m,\lambda} }
	{\sqrt{\lambda}}+1 \right)   \log\frac2\delta.
	$$
\end{lemma}

The following lemma, which can be found in \cite[eq. (48)]{Caponnetto2007}, gives an upper bound for $\mathcal P_{D,\lambda}, $ which measures the difference between
$L_{K,D}f^*$ and $S_D^Ty_D$. 

\begin{lemma}\label{Lemma:functional-difference}
	Let $\delta\in(0,1)$ and $\lambda>0$.  If $|y|\le M$  almost surely, then
	with confidence at least $1-\delta,$ there holds
	\[
	\mathcal P_{D,\lambda}\le
	\frac{2M}{\kappa}\mathcal B_{m,\lambda}\log\frac{2}{\delta}.
	\]
\end{lemma}

The following lemma gives an upper bound for  $\mathcal U_{D,\lambda}$; see \cite[Lemma 18]{Lin2017}

\begin{lemma}\label{Lemma:functional-difference-out}
	Let $f$be a bounded function. For any $0<\delta<1$ and $\lambda>0$, with confidence at least
	$1-\delta,$ there holds
	$$
	\mathcal U_{D,\lambda,f} \leq \frac{\|f\|_\infty}{\kappa}\mathcal B_{m,\lambda}\log\frac{2}{\delta},
	$$
	where $\|f\|_\infty:=\sup_{x\in\mathcal X}|f(x)|$.
\end{lemma}

We point out that all the above results require the presence of upper bounds of effective dimensions.
The following lemma
\cite[Corollary 2.2]{Mucke2018} (see also \cite[Lemma 21]{Lin2019}) 
features some relations between the effective dimension $\mathcal N(\lambda)$ and its empirical
counterpart $\mathcal N_D(\lambda)$.
\begin{lemma}\label{Lemma:Effective dimension}
	For any $0<\delta<1$ and $\lambda>0$, with
	confidence $1-\delta$, there holds
	\begin{eqnarray*}
		&&(1+4\eta_\delta)^{-1}
		\sqrt{\max\{\mathcal N(\lambda),1\}}
		\leq
		\sqrt{\max\{\mathcal N_D(\lambda),1\}}\\
		&\leq & (1+4\max\{\sqrt{\eta_\delta},\eta_\delta^2\})\sqrt{\max\{\mathcal
			N(\lambda),1\}},
	\end{eqnarray*}
	where $\eta_\delta:=2\log(4/\delta)/\sqrt{m\lambda}.$
\end{lemma}
From the above lemmas, we can deduce the following error estimates for operator differences in terms of the eigenvalues of $\mathbb K$.

\begin{lemma}\label{multi-log-esti}
	Let $\alpha >0$ and $d \ge 1$ be given. Then for any $0 < \lambda \le 1,$ we have
	\[
	\int^\infty_0 \frac{dt}{1+ \lambda e^{\alpha\ t^{1/d}}} \le C_1\  \frac{d!}{\alpha^d}\ \log^d \frac{1}{\lambda},
	\]
	where $C_1$ is an absolute constant.
\end{lemma}

\begin{proof}
	By substituting $u=\lambda e^{\alpha\ t^{1/d}}$, we have
	\[
	\int^\infty_0 \frac{dt}{1+ \lambda e^{\alpha\ t^{1/d}}} =
	\frac{d}{\alpha^d}\  \int^\infty_\lambda \frac{\log^{d-1}\frac{u}{\lambda}}{u(u+1)}du.
	\]
	We denote $L(\lambda)$ the integral on the right hand side of the above equation (without the constant $\frac{d}{\alpha^d}$). For an $M \ge \lambda,$ write 
	\[
	L(\lambda, M):=\int^M_\lambda \frac{\log^{d-1}\frac{u}{\lambda}}{u(u+1)}du
	=I_{0,0}(\lambda, M) - I_{1,1}(\lambda, M)+ I_{1,1}(\lambda, M) - I_{0,1}(\lambda, M),
	\]
	in which
	\[
	I_{i,j}(\lambda, M):= \int^M_\lambda \frac{\log^{d-1}\frac{u+i}{\lambda}}{u+j}du, \quad i,j = 0,1.
	\]
	For $d \ge 2$, we make use of the inequality $\log (1 + x) \le x \;( x \ge 0$) to write
	\begin{align*}
	0 \le & \log^{d-1}\frac{u+1}{\lambda}-\log^{d-1}\frac{u}{\lambda} \\
	\le & \frac1u \sum^{d-2}_{i=0} \log^{d-i-2}\frac{u}{\lambda}\ \log^{i}\frac{u+1}{\lambda} \\
	\le &  \frac{d-1}{u} \log^{d-2}\frac{u+1}{\lambda}.
	\end{align*}
	It follows that
	\[
	I_{1,1}(\lambda, M) - I_{0,1}(\lambda, M) \le (d-1)
	\int^\infty_\lambda \frac{\log^{d-2}\frac{u+1}{\lambda}}{u(1+u)}du.
	\]
	Note that the integral on the right hand side of the above inequality is of the order $\circ\left( L(\lambda)\right).$ It remains to estimate $I_{0,0}(\lambda, M) - I_{1,1}(\lambda, M)$. To this end, we write
	\begin{align*}
	&I_{0,0}(\lambda, M) - I_{1,1}(\lambda, M) \\
	=   & \log^d\frac{u+1}{\lambda}-\log^d \frac{u}{\lambda} \\
	\le & C \log^d \frac1\lambda - O \left( \frac1M \log^{d-1} \frac{M}{\lambda}\right).
	\end{align*}
	Letting $M \rightarrow \infty$, we get the desired result. A mathematical induction argument shows that the constant (depending on $d$) is of the order $d!$.
\end{proof}

\begin{lemma}\label{Lemma:Q}
	Let
	$\delta\in(0,1)$ and $\lambda>0$. If $\Lambda_D$ is identically and independently drawn according to $\rho_X$ and $|y|\le M$  almost surely, then with confidence $1-\delta$, there holds
	\begin{eqnarray*}
		\mathcal Q_{D,\lambda}
		&\leq&
		\sqrt{2}   (2\kappa(\kappa+8)\mathcal A_{D,\lambda}+1)    \log^2\frac8\delta,\\
		\mathcal W_{D,\lambda}
		&\leq&
		2\kappa(\kappa+8)\sqrt{\lambda}\mathcal A_{D,\lambda} \log^2 \frac8\delta,\\
		\mathcal P_{D,\lambda}
		&\leq&
		4M(\kappa+8)\sqrt{\lambda}\mathcal A_{D,\lambda}\log^2\frac{8}{\delta},\\
		\mathcal U_{D,\lambda,f}
		&\leq&
		2\frac{\|f\|_\infty(\kappa+8)}{\kappa}\sqrt{\lambda}\mathcal A_{D,\lambda}\log^2\frac{8}{\delta},
	\end{eqnarray*}
	where $\mathcal A_{D,\lambda}$ is defined by (\ref{Def.A}).
\end{lemma}

\begin{proof}
	From Lemma \ref{Lemma:Effective dimension}, (\ref{Definition B}) and (\ref{Effective-dimension}),   with confidence $1-\delta/2$, there holds
	\begin{eqnarray}\label{Bound.B}
	\mathcal B_{m,\lambda}
	\leq
	2\kappa \sqrt{\lambda}\mathcal A_{D,\lambda}\log\frac{8}\delta.
	\end{eqnarray}
	Plugging  (\ref{Bound.B}) into Lemma \ref{Lemma:Operator-product}, we have that
	$$
	\mathcal Q_{D,\lambda}\leq
	\sqrt{2}  \left( 2\kappa \mathcal A_{D,\lambda}+1 \right)   \log^2\frac8\delta
	$$
	holds with confidence $1-\delta$. The other bounds are derived similarly by inserting
	(\ref{Bound.B}) into Lemma \ref{Lemma:operator-difference}, Lemma \ref{Lemma:functional-difference} and Lemma \ref{Lemma:functional-difference-out}, respectively. This completes the proof of Lemma \ref{Lemma:Q}.
\end{proof}

\section*{Appendix C: Some Geometric Properties of the Random Sampling}
In this part of the article, we derive miscellaneous probabilistic and deterministic estimates for $q_{_\Xi}$.
\begin{lemma}\label{Lemma:mesh-separation}
	Let $\Xi:=\{x_i\}_{i=1}^m$ be i.i.d. drawn according to the uniform distribution on $\mathcal X$.   Then for any $t>0$, 
	\begin{eqnarray}\label{seperate-radius-prob}
	P(q_{_\Xi}\geq t)\geq 1-\frac{m^2\pi^{d/2}}{2{\rm Vol}{ (\mathcal X)}\Gamma(d/2+1)}t^d 
	\end{eqnarray}
	where ${\rm Vol}(\mathbb A)$ denotes the volume of the set $\mathbb A$ and $\Gamma$ is the Gamma function.
\end{lemma}

\begin{proof}  For each fixed $i=1,\dots,m$, let $B_t(x_i)$ be the ball with center $x_i$ and radius $t$. Let $E_i$ denote the event 
	that there is none  $j \ne i$ such that $x_j\in  B_t(x_i)$. We have that
	\[
	{\rm Vol.}(B_t(x_i)) = \frac{\pi^{d/2}}{\Gamma(d/2+1)}t^d,
	\]
	and therefore that
	\[
	P(E_i)= \left(1-\frac{\pi^{d/2}}{{\rm Vol.}{(\mathcal X)}\Gamma(d/2+1)}t^d\right)^{m-1}.
	\]
	It then follows that
	$$
	P(q_{_\Xi}\leq t)\leq P\left(\bigcup\limits^m_{i=1}E_i \right) \leq
	m\left[1-\left(1-\frac{\pi^{d/2}}{2{\rm Vol}{ (\mathcal X)}\Gamma(d/2+1)}t^d\right)^{m-1}\right].
	$$
	Using the inequality $(1-a)^b\geq 1-ba, \; \forall\ 0 \le a < 1, \; b \ge 0$, we derive
	$$
	P(q_{_\Xi}\leq t)\leq \frac{m^2\pi^{d/2}}{2{\rm Vol}{ (\mathcal X)}\Gamma(d/2+1)}t^d.
	$$
	This completes the proof of Lemma \ref{Lemma:mesh-separation}.  
\end{proof}

The result of the following lemma  is a direct consequence of Lemma A.2 and Corollary A.2 of \cite{Elkaroui2010}.
\begin{lemma}\label{Lemma:quadratic-type}
	Let $M>0$ and $\xi_i,\xi_j\in[-M,M]^d$ be  i.i.d. random vectors satisfying $E[\xi_i]=0$, $E[\xi_i^2]=\sigma^2$. 
	If
	\begin{equation}\label{d0}
	d>d_0:=\frac{2048\exp(4\pi)M^2}{\sigma^2},
	\end{equation}
	then we have
	\begin{eqnarray}\label{first}
	|\|\xi_i\|^2_2-d\sigma^2|\leq\frac{\sigma^2d}4,\qquad \mbox{and}\quad |\xi_i^T \xi_j|\leq\frac{\sigma^2d}4,\qquad \forall i,j=1,\dots,m,
	\end{eqnarray}
	with confidence at least
	\begin{eqnarray}\label{confidence}
	1-8\exp(4\pi)\left[\exp\left(-\frac{d}{96M^2}\right)+\exp\left(-\frac{ \sigma^2d}{5824M^2}\right)\right].
	\end{eqnarray}
	
\end{lemma}

\begin{lemma}\label{Lemma:separation-for-gaussian}
	Let $M>0$ and $\Xi=\{x_i\}_{i=1}^m\subset [-M,M]^d$ be a set of i.i.d. random vectors satisfying $E[x_i]=0$, $E[x_i^2]=\sigma^2$, $i=1,\dots,m$. Assume that inequality (\ref{d0}) holds true. Then we have
	\begin{equation}\label{separation-for-gaussian}
	q_{_\Xi} \geq  \frac{\sigma\sqrt{d}}2,
	\end{equation}
	with confidence at least
	\begin{equation}\label{confidence-Gaussian}
	1- 8m^2\exp(4\pi)\left[\exp\left(-\frac{d}{96M^2}\right)+\exp\left(-\frac{ \sigma^2d}{5824M^2}\right)\right].
	\end{equation}
\end{lemma}

\begin{proof}
	For each pair of $1\leq i \ne j\leq m$,   we have
	\begin{eqnarray}\label{separation-1}
	&&2d\sigma^2 -\|x_i-x_j\|_2^2
	=2d\sigma^2 -(x_i^T-x_j^T)(x_i-x_j)\nonumber\\
	&=&
	d\sigma^2  -\|x_i\|_2^2+ d\sigma^2 -\|x_j\|_2^2+x_i^T x_j+x_j^T x_i.
	\end{eqnarray}
	Plugging (\ref{first}) into (\ref{separation-1}), we obtain that
	$   2d\sigma^2 -\|x_i-x_j\|_2^2\leq  d\sigma^2,$
	for each pair of $i \ne j$, with confidence at least
	$$
	1-8\exp(4\pi)\left[\exp\left(-\frac{d}{96M^2}\right)+\exp\left(-\frac{ \sigma^2d}{5824M^2}\right)\right],
	$$
	which implies that with the same amount of confidence, we have
	$$
	\|x_i-x_j\|_2\geq  \sigma\ \sqrt{d}, \quad 1 \le i \ne j \le  m.
	$$
	Considering all such pairs of $1 \le i \ne j \le  m$, we derive
	that $ q_{_\Xi} \geq  {\displaystyle \frac{\sigma\sqrt{d}}2},
	$
	with confidence at least
	$$
	1- 8m^2\exp(4\pi)\left[\exp\left(-\frac{d}{96M^2}\right)+\exp\left(-\frac{ \sigma^2d}{5824M^2}\right)\right],
	$$
	which is the desired result.  
\end{proof}


\bibliographystyle{siamplain}
\bibliography{interpolation}
\end{document}

%% file: ex_shared.tex
\usepackage{amsmath}
\usepackage{amsfonts}
\usepackage{graphicx}
\usepackage{epstopdf}
\usepackage{comment}
\usepackage{subfigure}
\usepackage{algorithmic}
\ifpdf
  \DeclareGraphicsExtensions{.eps,.pdf,.png,.jpg}
\else
  \DeclareGraphicsExtensions{.eps}
\fi

\newcommand{\N}{\mathbb{N}}

\newsiamremark{remark}{Remark}
\newsiamremark{hypothesis}{Hypothesis}
\crefname{hypothesis}{Hypothesis}{Hypotheses}
\newsiamthm{claim}{Claim}

\headers{Kernel Interpolation of High Dimensional Scattered Data}{S.-B. Lin, X. Chang, and X.Sun}

\title{Kernel Interpolation of High Dimensional Scattered Data\thanks{\today.
\funding{The first two authors of the article are partially
	supported by the
	National Natural Science Foundation of China [Grant Nos. 61876133,11771012].}}}

\author{Shao-bo Lin\thanks{Center for Intelligent Decision-Making and Machine Learning, School of Management, Xi'an Jiaotong University, Xi'an 710049, China
  (\email{sblin1983@gmail.com}).}
\and Xiangyu Chang\thanks{Center for Intelligent Decision-Making and Machine Learning, School of Management, Xi'an Jiaotong University, Xi'an 710049, China
  (\email{xiangyuchang@xjtu.edu.cn}).}
\and Xingping Sun\thanks{ Department
	of Mathematics, Missouri State University, Springfield, MO 65897,
	USA
  (\email{XSun@MissouriState.edu}).}}

\usepackage{amsopn}

\newcommand{\Z}{ 
    \mathbb{Z}
}

\newcommand{\sfgrad}[1][]{ 
	\nabla_{*}
}

\newcommand{\sfcurl}[1][]{ 
	\mathbf{L}
}

\newcommand{\Jw}[1][\alpha,\beta]{ 
w_{#1}
}

\newcommand{\imat}[1][d]{ 
    I
}

\newcommand{\Lpw}[2][\Jw]{ 
\mathbb{L}_{#2}(#1)
}

\newcommand{\InnerLGb}[2][{\Jw[r-\frac{1}{2},r-\frac{1}{2}]}]{ 
\left(#2\right)_{\Lpw[{#1}]{2}}
}

\newcommand{\Diff}[2][t]{ 
\ifthenelse{\equal{#2}{1}}{\frac{\mathrm{d}}{\mathrm{d}#1}}{
\left(\frac{\mathrm{d}}{\mathrm{d}#1}\right)^{#2}}
}

\newcommand{\R}{\mathbb{R}}



%% file: interpolation-hd210422.bbl
\begin{thebibliography}{10}

\bibitem{aronszajn1950theory}
{\sc N.~Aronszajn}, {\em Theory of reproducing kernels}, Transactions of the
  American Mathematical Society, 68 (1950), pp.~337--404.

\bibitem{Baldi2011}
{\sc P.~Baldi and G.~W. Hatfield}, {\em DNA Microarrays and Gene Expression:
  from experiments to Data Analysis and Modeling}, Cambridge University Press,
  2011.

\bibitem{Ball1992}
{\sc K.~M. Ball}, {\em Invertibility of euclidean distance matrices and radial
  basis interpolation}, Journal of Approximation Theory, 68 (1992), pp.~74--82.

\bibitem{Belkin2018}
{\sc M.~Belkin}, {\em Approximation beats concentration? an approximation view
  on inference with smooth radial kernels}, arXiv preprint arXiv:1801.03437,
  (2018).

\bibitem{Bhatia2013}
{\sc R.~Bhatia}, {\em Matrix analysis}, vol.~169, Springer Science \& Business
  Media, 2013.

\bibitem{Blanchard2016}
{\sc G.~Blanchard and N.~Kr{\"a}mer}, {\em Convergence rates of kernel
  conjugate gradient for random design regression}, Analysis and Applications,
  14 (2016), pp.~763--794.

\bibitem{Blanchard2012}
{\sc G.~Blanchard and P.~Math{\'e}}, {\em Discrepancy principle for statistical
  inverse problems with application to conjugate gradient iteration}, Inverse
  Problems, 28 (2012), p.~115011.

\bibitem{Caponnetto2007}
{\sc A.~Caponnetto and E.~De~Vito}, {\em Optimal rates for the regularized
  least-squares algorithm}, Foundations of Computational Mathematics, 7 (2007),
  pp.~331--368.

\bibitem{Deng2009}
{\sc J.~Deng, W.~Dong, R.~Socher, L.-J. Li, K.~Li, and L.~Fei-Fei}, {\em
  Imagenet: A large-scale hierarchical image database}, in 2009 IEEE Conference
  On Computer Vision and Pattern Recognition, IEEE, 2009, pp.~248--255.

\bibitem{fuselier-hangel-narc-ward}
{\sc E.~Fuselier, T.~Hangelbroek, F.~J. Narcowich, J.~D. Ward, and G.~B.
  Wright}, {\em Localized bases for kernel spaces on the unit sphere}, SIAM
  Journal on Numerical Analysis, 51 (2013), pp.~2538--2562.

\bibitem{Guo2017}
{\sc Z.-C. Guo, S.-B. Lin, and D.-X. Zhou}, {\em Learning theory of distributed
  spectral algorithms}, Inverse Problems, 33 (2017), p.~074009.

\bibitem{hall2005geometric}
{\sc P.~Hall, J.~S. Marron, and A.~Neeman}, {\em Geometric representation of
  high dimension, low sample size data}, Journal of the Royal Statistical
  Society: Series B (Statistical Methodology), 67 (2005), pp.~427--444.

\bibitem{hangel-narc-rieger-ward}
{\sc T.~Hangelbroek, F.~J. Narcowich, C.~Rieger, and J.~D. Ward}, {\em Direct
  and inverse results on bounded domains for meshless methods via localized
  bases on manifolds}, in Contemporary Computational Mathematics-A Celebration
  of the 80th Birthday of Ian Sloan, Springer, 2018, pp.~517--543.

\bibitem{hangel-narc-sun-ward}
{\sc T.~Hangelbroek, F.~J. Narcowich, X.~Sun, and J.~D. Ward}, {\em Kernel
  approximation on manifolds \text{II}: The $\ell_{\infty}$ norm of the
  $\ell_2$ projector}, SIAM Journal on Mathematical Analysis, 43 (2011),
  pp.~662--684.

\bibitem{hangel-narc-ward-1}
{\sc T.~Hangelbroek, F.~J. Narcowich, and J.~D. Ward}, {\em Kernel
  approximation on manifolds \text{I}: bounding the lebesgue constant}, SIAM
  Journal on Mathematical Analysis, 42 (2010), pp.~1732--1760.

\bibitem{hangel-narc-ward-2}
{\sc T.~Hangelbroek, F.~J. Narcowich, and J.~D. Ward}, {\em Polyharmonic and
  related kernels on manifolds: interpolation and approximation}, Foundations
  of Computational Mathematics, 12 (2012), pp.~625--670.

\bibitem{hesse2017radial}
{\sc K.~Hesse, I.~H. Sloan, and R.~S. Womersley}, {\em Radial basis function
  approximation of noisy scattered data on the sphere}, Numerische Mathematik,
  137 (2017), pp.~579--605.

\bibitem{Elkaroui2010}
{\sc E.~Karoui}, {\em The spectrum of kernel random matrices}, The Annals of
  Statistics, 38 (2010), pp.~1--50.

\bibitem{kuhn2}
{\sc T.~K{\"u}hn}, {\em Eigenvalues of integral operators with smooth positive
  definite kernels}, Archiv der Mathematik, 49 (1987), pp.~525--534.

\bibitem{levesley-sun}
{\sc J.~Levesley and X.~Sun}, {\em Approximation in rough native spaces by
  shifts of smooth kernels on spheres}, Journal of Approximation Theory, 133
  (2005), pp.~269--283.

\bibitem{Liang2020}
{\sc T.~Liang and A.~Rakhlin}, {\em Just interpolate: Kernel ``ridgeless''
  regression can generalize}, Annals of Statistics, 48 (2020), pp.~1329--1347.

\bibitem{Lin2017}
{\sc S.-B. Lin, X.~Guo, and D.-X. Zhou}, {\em Distributed learning with
  regularized least squares}, The Journal of Machine Learning Research, 18
  (2017), pp.~3202--3232.

\bibitem{Lin2019}
{\sc S.-B. Lin, Y.~Lei, and D.-X. Zhou}, {\em Boosted kernel ridge regression:
  Optimal learning rates and early stopping.}, Journal of Machine Learning
  Research, 20 (2019), pp.~1--36.

\bibitem{lin2021distributed}
{\sc S.-B. Lin, Y.~G. Wang, and D.-X. Zhou}, {\em Distributed filtered
  hyperinterpolation for noisy data on the sphere}, SIAM Journal on Numerical
  Analysis, 59 (2021), pp.~634--659.

\bibitem{Lin2018}
{\sc S.-B. Lin and D.-X. Zhou}, {\em Distributed kernel-based gradient descent
  algorithms}, Constructive Approximation, 47 (2018), pp.~249--276.

\bibitem{madych-nelson}
{\sc W.~Madych and S.~Nelson}, {\em Multivariate interpolation and
  conditionally positive definite functions \text{II}}, Mathematics of
  Computation, 54 (1990), pp.~211--230.

\bibitem{Mucke2018}
{\sc N.~M{\"u}cke}, {\em Adaptivity for regularized kernel methods by lepskii's
  principle}, arXiv preprint arXiv:1804.05433,  (2018).

\bibitem{Narcowich2007}
{\sc F.~J. Narcowich, X.~Sun, J.~D. Ward, and H.~Wendland}, {\em Direct and
  inverse sobolev error estimates for scattered data interpolation via
  spherical basis functions}, Foundations of Computational Mathematics, 7
  (2007), pp.~369--390.

\bibitem{Narcowich1991}
{\sc F.~J. Narcowich and J.~D. Ward}, {\em Norms of inverses and condition
  numbers for matrices associated with scattered data}, Journal of
  Approximation Theory, 64 (1991), pp.~69--94.

\bibitem{Narcowich2002}
{\sc F.~J. Narcowich and J.~D. Ward}, {\em Scattered data interpolation on
  spheres: error estimates and locally supported basis functions}, SIAM Journal
  on Mathematical Analysis, 33 (2002), pp.~1393--1410.

\bibitem{Narcowich2004}
{\sc F.~J. Narcowich and J.~D. Ward}, {\em Scattered-data interpolation on
  $\mathbb{R}^n$: Error estimates for radial basis and band-limited functions},
  SIAM Journal on Mathematical Analysis, 36 (2004), pp.~284--300.

\bibitem{Narcowich2006}
{\sc F.~J. Narcowich, J.~D. Ward, and H.~Wendland}, {\em Sobolev error
  estimates and a bernstein inequality for scattered data interpolation via
  radial basis functions}, Constructive Approximation, 24 (2006), pp.~175--186.

\bibitem{Park1991}
{\sc J.~Park and I.~W. Sandberg}, {\em Universal approximation using
  radial-basis-function networks}, Neural computation, 3 (1991), pp.~246--257.

\bibitem{peetre}
{\sc J.~Peetre}, {\em A Theory of Interpolation of Normed Spaces}, vol.~39,
  Instituto de Matem{\'a}tica Pura e Aplicada, Conselho Nacional de Pesquisas,
  1968.

\bibitem{Pinelis1994}
{\sc I.~Pinelis}, {\em Optimum bounds for the distributions of martingales in
  \text{Banach} spaces}, The Annals of Probability,  (1994), pp.~1679--1706.

\bibitem{reade-2}
{\sc J.~Reade}, {\em Eigenvalues of positive definite kernels}, SIAM Journal on
  Mathematical Analysis, 14 (1983), pp.~152--157.

\bibitem{Schaback1995}
{\sc R.~Schaback}, {\em Error estimates and condition numbers for radial basis
  function interpolation}, Advances in Computational Mathematics, 3 (1995),
  pp.~251--264.

\bibitem{Schaback2000}
{\sc R.~Schaback}, {\em A unified theory of radial basis functions: Native
  hilbert spaces for radial basis functions \text{II}}, Journal of
  computational and applied mathematics, 121 (2000), pp.~165--177.

\bibitem{Smale2004}
{\sc S.~Smale and D.-X. Zhou}, {\em Shannon sampling and function
  reconstruction from point values}, Bulletin of the American Mathematical
  Society, 41 (2004), pp.~279--305.

\bibitem{Smale2005}
{\sc S.~Smale and D.-X. Zhou}, {\em Shannon sampling \text{II}: Connections to
  learning theory}, Applied and Computational Harmonic Analysis, 19 (2005),
  pp.~285--302.

\bibitem{Smale2007}
{\sc S.~Smale and D.-X. Zhou}, {\em Learning theory estimates via integral
  operators and their approximations}, Constructive Approximation, 26 (2007),
  pp.~153--172.

\bibitem{SteinwartHS}
{\sc I.~Steinwart, D.~R. Hush, and C.~Scovel}, {\em Optimal rates for
  regularized least squares regression.}, in COLT, 2009, pp.~79--93.

\bibitem{Wendland2004}
{\sc H.~Wendland}, {\em Scattered Data Approximation}, vol.~17, Cambridge
  university press, 2004.

\bibitem{weyl}
{\sc H.~Weyl}, {\em Das asymptotische verteilungsgesetz der eigenwerte linearer
  partieller differentialgleichungen (mit einer anwendung auf die theorie der
  hohlraumstrahlung)}, Mathematische Annalen, 71 (1912), pp.~441--479.

\bibitem{Yao2007}
{\sc Y.~Yao, L.~Rosasco, and A.~Caponnetto}, {\em On early stopping in gradient
  descent learning}, Constructive Approximation, 26 (2007), pp.~289--315.

\end{thebibliography}
